\documentclass[preprint,12pt]{elsarticle}

\RequirePackage{amsthm,amsmath}
\RequirePackage[colorlinks,citecolor=blue,urlcolor=blue]{hyperref}
\usepackage{amssymb,dsfont}
\usepackage{graphicx} 
\usepackage[utf8]{inputenc}
\usepackage{subcaption} 
\usepackage[normalem]{ulem}
\usepackage{mathtools}
\mathtoolsset{showonlyrefs}

\newcommand{\E}{\mathbb{E}}

\newcommand{\hz}{{\widehat{\mathbf{z}}_n}}

\newcommand{\n}[2]{n_{#1#2}}
\newcommand{\na}[1]{n_{#1}}
\renewcommand{\o}[2]{o_{#1#2}}
\newcommand{\Q}[1]{Q_{\textsc{ml}}(#1)}
\newcommand{\Qb}[1]{Q_{\textsc{icl}}(#1)}

\newcommand{\1}{\mathbf{1}}
\newcommand{\diag}{\text{Diag}}
\newcommand{\one}{\mathbf{1}}
\newcommand{\z}{\mathbf{z}_n}
\newcommand{\e}{\mathbf{e}_n}

\newcommand{\Z}{{\mathbf{Z}_n}}

\newcommand{\A}{\mathbf{A}_{n\times n}}
\renewcommand{\a}{\mathbf{a}_{n\times n}}

\newcommand{\smi}{s_{\min}}
\newcommand{\sma}{s_{\max}}
\newcommand{\F}{\mathcal F(n,\alpha)}
\renewcommand{\P}{\mathbb{P}}

\DeclareMathOperator*{\argmax}{arg\,max}  

\numberwithin{equation}{section}
\theoremstyle{plain}
\newtheorem{theorem}{Theorem}[section]
\newtheorem{lemma}[theorem]{Lemma}
\newtheorem{corollary}[theorem]{Corollary}
\newtheorem{proposition}[theorem]{Proposition}
\newtheorem{remark}[theorem]{Remark}

\begin{document}

\begin{frontmatter}



\title{Optimal recovery by maximum and integrated conditional likelihood in the general Stochastic Block Model\footnote{In memory of Antonio Galves, whose insatiable curiosity and generous guidance will forever inspire our work.
}} 


\author[1]{Andressa Cerqueira}
\author[2]{Florencia Leonardi}

\affiliation[1]{organization={Departament of Statistics, Universidade Federal de S\~ao Carlos},
            city={São Carlos},
            state={São Paulo},
            country={Brazil}}
\affiliation[2]{organization={Department of Statistics, Universidade de S\~ao Paulo},
            city={São Paulo},
            state={São Paulo},
            country={Brazil}}

\begin{abstract}
In this paper, we obtain new results on the weak and strong consistency of the maximum and integrated conditional likelihood estimators for the community detection problem in the Stochastic Block Model with $k$ communities and unknown parameters. In particular we show that maximum conditional likelihood achieves the optimal known threshold for exact 
recovery in the logarithmic degree regime. For the integrated conditional likelihood, we obtain a sub-optimal constant, but still obtain strong consistency in the logarithmic degree regime. Both methods are shown to be weakly consistent in the divergent degree regime. These results  fill in the gap in the theory of community detection with maximum likelihood and integrated conditional likelihood, solving open problems in the literature.
\end{abstract}

%


\end{frontmatter}

\section{Introduction}

Network analysis has become an active area of research aimed at describing and model random interactions between pairs of actors in a population.  The network's nodes represent the actors, while the edges represent the interactions between them. Network analysis is applied nowadays to several fields, such as social sciences, biology, computer science, neuroscience, and many others. Community detection is an important task in network analysis, and it is related to identifying groups of nodes in an observed network based on their connectivity patterns. 

The Stochastic Block Model (SBM) proposed by \citet{holland1983stochastic} is a statistical model for networks that incorporates the community structure of the nodes. In an SBM, the nodes are partitioned randomly into communities, and the existence of an edge between pairs of nodes depends on their community membership. More precisely, the existence of the edges in the network is modeled by independent Bernoulli random variables, conditionally on the nodes' memberships, with parameters depending only on the community assignments. In this way, the SBM can be seen as a latent variable model, as only the edges are observed, and its statistical analysis is related to that of the perturbed chains with unbounded variable memory \cite{collet-etal2008}.

In the recent years, the classical SBM has been generalised to consider other aspects of the network, such as heterogeneity of the node's degree \citep{karrer2011stochastic, cerqueira2024}, nodes belonging to more than one community \citep{latouche2011overlapping,airoldi2008mixed}, networks that evolve on time \citep{hero2014dynamic,matias2017statistical}. But regardless of the several extensions, and due to its simplicity, much recent research has focused on the recovery of the communities for the classical SBM. An incomplete list includes Newman-Girvan modularity \citep{newman-girvan2002}, maximum likelihood-based methods \citep{bickel2009nonparametric,amini2013pseudo}, spectral clustering \citep{lei-rinaldo2015,mossel2016}, semidefinite programming relaxation of maximum likelihood \citep{abbe2015exact} and  a Bayesian estimator \citep{van2018bayesian}. 

Given the large number of community recovery methods proposed in the literature, it is natural to inquire about necessary and sufficient conditions guaranteeing the recovery of the true community assignments. The recovery of the communities has been studied in two main aspects: almost exact and exact recovery. Almost exact  recovery (also called weak consistency) means that an increasing fraction of the node's communities is recovered, with a high probability, when the network size increases. In contrast, exact recovery (strong consistency) means that all the nodes' communities are recovered with a high probability when the network size increases. These notions are defined up to some permutation of the true labels, as exact determination of the node's labels is impossible. 

In recent works, \citet{abbe2015exact} and \citet{mossel2016}  established theoretical conditions and thresholds for  almost exact and exact community recovery in the two symmetric communities model, also known as planted bisection model, and propose algorithms attaining these thresholds. The planted bisection model is a simplified version of SBM, consisting of two equal-sized communities and two parameters: $\rho_n s_1$ for the within-community edge probability and $\rho_n s_2$ for the between-community edge probability, where $\rho_n \to 0$, controls the sparsity of the graph as the number of nodes increases and $s_1,s_2>0$. They assume the model's parameters are known. For almost exact recovery, the necessary and sufficient condition is that $n\rho_n\to\infty$ as $n\to\infty$ \citep{mossel2016}. For exact recovery, both works established a sharp phase transition threshold when $\rho_n = \log n / n$, which corresponds to the minimal degree regime required for strong consistency, under the condition that
\[
(\sqrt{s_1} - \sqrt{s_2})^2 \geq 2.
\]

For the general SBM with $k$ communities and known parameters, \citet{abbe-colin2015} proved a phase transition phenomenon under the regime $\rho_n = \log n / n$, which is characterized by the Chernoff–Hellinger (CH) divergence. They showed that exact recovery is possible when the minimum of the CH divergence is  greater than or equal to one. Differing from earlier works, \citet{NIPS2015_abbe_sandon} consider the problem of recovering communities when the model parameters are unknown. They propose an algorithm that achieves the optimal information-theoretic threshold characterized by the CH divergence, as established in \citet{abbe-colin2015}.

From a statistical perspective, many estimators are, in general, derived using approaches based on maximum likelihood and Bayesian methods. These methods can similarly be applied to community recovery, with the key advantage that they do not require prior knowledge of the model parameters.  However, concerning maximum likelihood and Bayesian estimators, the strong consistency results for community detection known so far have been relatively limited. In particular, \citet{bickel2009nonparametric} proved the 
strong consistency of the maximum likelihood estimator in the regime $\rho_n \gg \log n/n$. 
From the Bayesian  perspective, \citet{van2018bayesian} proved that the posterior mode is a 
strongly consistent estimator of the communities when $\rho_n \gg (\log n)^2/n$. Even though the 
latter work has a stronger condition than  \citet{bickel2009nonparametric}, the authors pointed out 
that the former work assumed a global Lipschitz condition that does not generally hold. For that 
reason, and by the relation of both estimators, up to this moment,  the strong consistency of the 
maximum likelihood and Bayesian estimators are only verified in the regime $\rho_n \gg (\log n)^2/
n$. Moreover, in these works, the optimality of the maximum likelihood and Bayesian estimators has not been established.

In this work, we study the maximum and integrated conditional likelihood estimators for the general 
SBM with $k$ communities not necessarily of the same size and without assuming known 
parameters. The integrated conditional likelihood estimator is a modified version of the Bayesian estimator proposed in \citet{van2018bayesian}, attaining an almost optimal rate of convergence. We prove that weak consistency can be obtained for both methods under the same necessary regime  $n\rho_n\to\infty$ as $n\to\infty$ established by \citet{mossel2016}. Moreover, we prove that the maximum and integrated conditional likelihood estimators are strongly consistent under the regime  $\rho_n\geq \log n/n$, with a condition on the parameters. 
In the case $\rho_n = \log n / n$, we establish the optimality of the maximum likelihood estimator under the same condition characterized by the CH divergence, as introduced by \citet{abbe-colin2015}.  For the integrated conditional likelihood estimator, we obtain a sub-optimal condition on the constant, but still prove the consistency under the phase transition regime $\rho_n = \log n/n$. The proof of the strong consistency of these estimators is derived using a new concentration result based on the general Chernoff bound. With these results, we show that the maximum and integrated likelihood estimators are weakly and strongly consistent in both optimal and sub-optimal regimes, extending the contributions of  \citet{bickel2009nonparametric} and  \citet{van2018bayesian}. 

This paper is organized as follows. In Section~\ref{defs} we introduce the main definition of the model and of the estimators, and state a result relating both estimators.  In Section~\ref{sec:consistency} we state the main results concerning weak and strong consistency. An empirical analysis of the performance of both estimators is presented in Section~\ref{sec:empirical}. In Section~\ref{proofs}, we present the proofs of the main theorems related to the strong and weak consistency of these estimators. Section~\ref{discussion} presents a brief discussion, and in the Appendix, we state and prove several auxiliary results, including the new concentration result based on the Chernoff bound. 

\section{Community detection in Stochastic Block Models}\label{defs}

The SBM with $n$ nodes and $k$ communities can be described by the pair of random structures $(\Z,\A)$, where $\Z$ denotes the community assignment of the nodes and $\A$ is the adjacency matrix representing the network. The random vector $\Z=(Z_1,Z_2,\dots,Z_n)$ of community assignments is composed by independent and identically distributed random variables with distribution $\pi$ over $[k]=\{1,2,\dots,k\}$. 
Given $\Z=\z$, the law of the adjacency matrix $\A$ is a product  of Bernoulli random variables whose parameters depend only on the nodes' labels. More formally, there exists a symmetric probability matrix  $P \in [0,1]^{k\times k}$ such that 
the conditional distribution of $\A=\a$ given $\Z=\z$ can be written as 
\begin{equation}\label{def-prob}
\P(\a|\z) = 
\prod_{1\leq a, b\leq k} \!\!P_{ab}^{\o{a}{b}/2} (1-P_{ab})^{(\n{a}{b}-\o{a}{b})/2}\,,
\end{equation}
where the counters $n_a=n_a(\z)$, $\n{a}{b}=\n{a}{b}(\z)$ and $\o{a}{b}=\o{a}{b}(\z,\a)$ are given by
\begin{align*}
n_a(\z) &= \sum\limits_{i=1}^n \mathds{1}\{z_i=a\}\, , \qquad\quad\;\, 1 \leq a \leq k\,,\\
\n{a}{b}(\z) &=\begin{cases}
n_a(\z)n_b(\z)\, ,& 1 \leq a < b \leq k,\\
n_a(\z)(n_a(\z)-1)\,, & 1 \leq a=b \leq k,\,
\end{cases}
\end{align*}
and  
\[
\o{a}{b}(\z,\a) =  \sum\limits_{1\leq i, j\leq n} \mathds{1}\{z_i=a,z_j=b\}a_{ij}\,.
\]

As it is usual in the definition of likelihood functions, by convention, we define $0^0=1$ in \eqref{def-prob} when some of the parameters are 0. 

For any $a\in[k]$, the expected number of nodes in community $a$ is given by $n\pi_a$. When $\pi_a=1/k$ for all $a\in[k]$, we say that the network is \textit{balanced}. However, since $\Z$ is randomly generated, the community sizes may differ even in the balanced case.
Given a matrix $P$ and a node in the community $a$, its expected degree is given by $n\sum_{b}\pi_bP_{ab}$. Then, assuming $k$ and $\pi$ fixed, the expected degree of a node scales on the order of $n$ times the connection probabilities given by $P$.
When $P$ is fixed with bounded entries, the expected degree grows linearly with $n$ and the network is \emph{dense}. However, in general, observed networks do not have too many edges, so the interesting regimes are those where $P$ decreases with the number of nodes $n$ and the expected degree is of order much smaller than $n$.
To emphasize the dependence on $n$, we reparametrize $P$ by writing $P = \rho_n S$, where $S$ is a matrix with all entries strictly positive.  For simplicity in the presentation of the results, we assume all entries in $P$ have the same asymptotic order $\rho_n$, so we can define the values
\[
\smi = \min_{a,b} S_{ab}\quad \text{ and }\quad \sma = \max_{a,b} S_{ab}
\]
and take them as strictly positive constants not depending on $n$. 
We are interested in both dense and sparse networks. In particular, we consider the dense case, where $\rho_n = 1$, and the sparse case, where the sequence $\{\rho_n\}_{n \in \mathbb{N}}$ satisfies $\rho_n \to 0$ as $n \to \infty$.

In this paper, we obtain conditions over $(\pi, \rho_nS)$ in order to obtain weak and strong consistency of the maximum and integrated conditional likelihood estimators.  To guarantee the identifiability of the model, we make the assumption that all entries of $\pi$ are positive and that the matrix $S$ does not contain any two identical columns, as it is usually assumed in the literature. 

Given any community labels $\z\in[k]^n$, the maximum likelihood estimators for $P_{ab}$ can be easily obtained by maximizing \eqref{def-prob},  and they are given by $\o{a}{b}(\z,\a)/\n{a}{b}(\z)$. Using these estimators, we can write the \emph{likelihood modularity}  for any $\z\in[k]^n$ as 
\begin{equation}\label{def-P-hat}
\Q{\z} = \frac1{2n^2}\sum_{1\leq a, b\leq k} \n{a}{b}(\z) \tau\Bigl(  \dfrac{\o{a}{b}(\z,\a)}{\n{a}{b}(\z)}\Bigr)\,
\end{equation}
with $\tau(x) = x\log x + (1-x)\log (1-x)$. 
Given the observed graph $\a$, to estimate the community labels in $[k]^n$ we define the maximum likelihood (ML) estimator by
\begin{equation}\label{zstar}
\hz^{\textsc{ml}}= \argmax\limits_{\z \in \mathcal F(n,\alpha)}\left\lbrace \Q{\z} \right\rbrace\,,
\end{equation}
where, for some $\alpha>0$
\begin{equation}\label{calF}
\mathcal F(n,\alpha) =\{\z\in [k]^n\colon n_a(\z)\geq \alpha n \text{ for all }a\in [k]\}\,.
\end{equation}

The set $\mathcal F(n,\alpha)$ contains all community assignments such that the smallest community has a size at least $\alpha n$. This allows us to consider the community detection problem in the case of \textit{unbalanced} communities, that is, communities with different sizes. For networks with balanced communities, we have $\alpha = 1/k$. In this work, we assume $\alpha$ positive with the only restriction that $\alpha < \min_a \pi_a$. 
It is worth noting that requiring communities to have linear size does not seem too restrictive, since sublinear-sized communities would eventually ``disappear'' and thus would not fit within a fixed $k$-community model. On the other hand, this restriction appears as a necessary condition when considering the logarithm degree regime, even for weak consistency, see Theorem~\ref{teo_basico1}.

From a statistical point of view, another related approach to estimate the communities, from a Bayesian perspective, is the integrated conditional likelihood method. The integrated conditional likelihood is obtained by integrating out the parameter $P$ from the conditional distribution in \eqref{def-prob}, that is,
\begin{equation}
\mathbb Q(\a\mid \z) = \int_{[0,1]^{k(k+1)/2}}\P(\a\mid \z)\nu(P)d P\,,
\end{equation}
where $\nu(P)$ denotes a {\it prior} distribution over the matrix $P$. 
For convenience in the computations (see below), we assume that the prior distribution of $P$ is given by
\begin{equation}
\begin{split}
P_{ab} &\;\overset{i.i.d.}{\sim}\; \text{Beta}(1/2,1/2), \quad 1\leq a \leq b \leq k\,.
\end{split}
\end{equation}
Rewriting the likelihood \eqref{def-prob} by taking the product only on the communities $a$ and $b$ such that $a \leq b$ and using the conjugacy between the Bernoulli and Beta distributions, we obtain that 
\begin{equation}
\mathbb Q(\a\mid \z) = \prod_{1\leq a\leq b\leq k} \frac{ B( \tilde o_{ab}(\z,\a)+1/2, \tilde n_{ab}(\z)-\tilde o_{ab}(\z,\a)+1/2)}{B(1/2,1/2)}\,,
\end{equation}
where $\tilde o_{ab}(\z,\a)=o_{ab}(\z,\a)/2$, if $a= b$ and $\tilde o_{ab}(\z,\a)= o_{ab}(\z,\a)$ if $a\neq b$ and $\tilde n_{ab}(\z)= n_{ab}(\z)/2$, if $a=b$ and $\tilde n_{ab}(\z)=n_{ab}(\z)$ if $a\neq b$ and $B(x,y)$ denotes the beta-function.
Thus, the integrated conditional likelihood modularity can be defined as
\begin{equation}\label{def-P-hat2}
\Qb{\z} = \frac1{n^2}\sum_{1\leq a \leq b\leq k} \log \frac{B( \tilde o_{ab}(\z,\a)+1/2, \tilde n_{ab}(\z)-\tilde o_{ab}(\z,\a)+1/2)}{B(1/2,1/2)}\,.
\end{equation}
Given the observed graph $\a$, the estimator of the communities by the maximum integrated conditional likelihood (ICL) is given by 
\begin{equation}\label{zstar2}
\hz^{\textsc{icl}}\;=\; \argmax\limits_{\z \in \mathcal F(n,\alpha)} \Qb{\z} \,,
\end{equation}
where $\mathcal F(n,\alpha)$ is defined, for some $\alpha>0$, by \eqref{calF}. It is important to note that the ICL estimator proposed here differs from the Integrated Complete Likelihood, also called ICL, proposed by \citet{daudin2008mixture} for model selection: in our case, we integrate the conditional likelihood, whereas in their approach, the complete likelihood is integrated.

The first to propose a Bayesian community detection approach and to establish its  consistency (weak and strong)
was \citet{van2018bayesian}, under the condition $\rho_n \gg \log^2 n/n$. They defined prior distributions on $\pi$ and $P$ and derived the joint distribution of $\A$ and $\Z$ by integrating out the parameters.  The estimator of the communities is defined as the posterior mode of the distribution of $\Z$ given $\A$. In this work, we consider the integrated conditional likelihood approach, which is obtained by integrating out the parameter $P$ from the conditional distribution of $\A$ given $\Z$, and defining the estimator of the communities according to this distribution. This minor modification allows us to prove the weak consistency of this Bayesian estimator under the regime $n\rho_n \to\infty$ ad $n\to\infty$ and the strong consistency under the regime $\rho_n \geq \log n/n$. 

A relation between the ICL and the ML modularities, which allows us to apply the same proof techniques for community recovery to both estimators, is stated in the following lemma.

\begin{lemma}\label{lemma-qml-qb}
For all $n\geq 1$ we have that 
\[
\max_{\z\in[k]^n} \bigl| \;\Qb{\z} - \Q{\z} \;\bigr| \;\leq\; \frac{k^2(\log n+2)}{n^2}\,.
\]
\end{lemma}

The proof of this result is given in \ref{sec:modularities_relation}.

\section{Weak and strong consistency}\label{sec:consistency}
 
To prove the consistency results, we begin by defining a measure of discrepancy between two community assignments. Denote the discrepancy between any estimated community vector $\hz$ and the true community labels $\Z$ by the function
\begin{equation}\label{def:loss}
L(\hz,\Z) \;=\;  \min_{\sigma}\,  \Bigl\{\, \sum_{i=1}^n \1\{\hat z_i \neq  \sigma(Z_i)\} \,\Bigr\}
\end{equation}
where $\sigma$ denotes any permutation of the labels, $\sigma\colon [k]\to[k]$.

We say that an estimator is \textit{weakly consistent} if the fraction of misclassified nodes, up to a permutation, converges to zero; that is,  if for all $\epsilon >0$ we have 
\[
\P(L(\hz,\Z) < \epsilon n) \to 1\quad \text{ as } n\to\infty\,.
\]
A stronger requirement defines a \textit{strongly consistent} estimator if
asymptotically no errors occur, that is if
\[
\P(L(\hz,\Z) =0 ) \to 1\quad \text{ as } n\to\infty\,.
\]

Weak and strong consistency are related to different notions of recovery: weak consistency corresponds to almost exact recovery, while strong consistency corresponds to exact recovery \citep{abbe_jmlr2018}.  

Our first result refers to the weak consistency of the ML and ICL estimators. The weak consistency of these estimators is achieved under a more general sparse setting than that required for strong consistency, requiring only that $n\rho_n \to \infty$. This condition has been proven to be a necessary condition for weak consistency in the case of the symmetric two-community SBM with known parameters by \citet{mossel2016}.  To our knowledge, the weak consistency of the ML and the Bayesian estimator was established before under the condition $\rho_n\gg \log^2 n / n$ in \citet{van2018bayesian}, so Theorem~\ref{teorema-weak} fills in the gap attaining the optimal rate.   We also obtain a condition on the minimum size of a candidate community as a function of $n$, which still guarantees the convergence of the estimators. This is an interesting property that could enable $k$ to grow with the sample size, at a rate $k \asymp \alpha^{-1}$.

\begin{theorem}\label{teorema-weak}
Let $(\Z,\A)$ be generated from a SBM with parameters $(k,\pi, \rho_nS)$. Then 
if $\alpha<\min_{a} \pi_a$ and $n\rho_n\to\infty$ as $n\to\infty$, we have for both $\hz=\hz^\textsc{ml}$ and $\hz=\hz^\textsc{icl}$, that  for all $\epsilon >0$, $L(\hz,\Z) < \epsilon n$ eventually almost surely as $n\to\infty$.  Moreover, the convergence still holds if $\alpha$ is taken as decreasing with $n$ at a rate $\alpha \gg  (n\rho_n)^{-1/2}$. 
\end{theorem}

 We next present the strong consistency result for the estimators $\hz^\textsc{ml}$ and  $\hz^\textsc{icl}$  for $\rho_n  \geq \log n/n$. When $\rho_n \gg \log n/n$, any pair $(\pi, S)$ satisfies the conditions for strong consistency. However, in the critical case $\rho_n = \log n/n$, there is a precise phase transition threshold, depending on $(\pi, S)$, such that strong consistency of $\hz^\textsc{ml}$ holds whenever this threshold is exceeded. 
\citet{abbe-colin2015} studies the exact recovery
threshold in a general SBM with $k$ communities and known parameters. In this work, the authors defined the community profile associated with community $b$ by the vector $\theta_b$ with $\theta_{ba} = \pi_aS_{ab}$, for all $a,b\in [k]$.
Based on this definition, the authors prove that the accuracy of identifying the underlying communities increases as the separation between community profiles grows. To quantity the degree of separation of the community profiles they proposed the following function
\begin{equation}
D_{+}(\theta_{b'},\theta_{b}) = \max_{t\in[0,1]}\sum_{a} \pi_a( tS_{ab'}+(1-t)S_{ab}-S_{ab'}^{t}S_{ab}^{1-t}).
\end{equation}
For a fixed value of \(t\), this function defines an \(f\)-divergence with \(f(x) = 1 - t + tx - x^t\). Since it generalizes both the Hellinger divergence (recovered when \(t = 1/2\)) and the Chernoff divergence, the authors refer to it as the \textit{Chernoff-Hellinger (CH) divergence} (see Section 2.3 in \citet{abbe-colin2015} for further discussion). The results for the strong consistency of the estimators $\hz^\textsc{ml}$ and 
$\hz^\textsc{icl}$, given in Theorems \ref{teorema-chave} and \ref{teorema-chave-icl}, respectively, are associated with the CH divergence through the constant.
\begin{equation}\label{constantC}
C(\pi, S) \;=\;\min_{b\neq b'} D_{+}(\theta_{b'},\theta_{b}).
\end{equation}

\citet{abbe-colin2015} proved that exact recovery is information-theoretically solvable if and only if $C(\pi, S)>1$ when $\rho_n=\log n/n$. At the threshold, $C(\pi, S)=1$, exact recovery is solvable if and only if $S$ has no two equal columns and all its entries are non-zero
 
For the maximum likelihood estimator $\hz^\textsc{ml}$, when $\rho_n = \log n / n$, we have shown that strong consistency is achieved under the $C(\pi,S) > 1$. This generalizes the results of \citet{abbe-colin2015} and \citet{mossel2016} to the setting of classical maximum likelihood estimators in statistics.

\begin{theorem}\label{teorema-chave}
Let $(\Z,\A)$ be generated from a SBM with parameters $(\pi, \rho_nS)$, with $S$ a bounded symmetric matrix  with no two equal columns. If  $\log n/n \ll \rho_n\leq 1$ we have that  
\begin{equation*}
\P\Bigl(\, L(\hz^\textsc{ml},\Z) \, = \,0\, \Bigr) \;\to\; 1
\end{equation*}
as $n\to\infty$. 
The result also holds at the regime $\rho_n=\log n/n$ and $C(\pi,S)> 1$. 
\end{theorem}

\begin{remark}\label{mainremark}
\it{In the two symmetric communities model with  $\rho_n=\log n / n$ and 
\[
S = \begin{pmatrix} s_1 & s_2\\
s_2 & s_2
\end{pmatrix}
\]
the CH divergence is maximized at $t=1/2$ and we obtain
\begin{align*}
C(\pi,S)&
=  \frac12(\sqrt{s_1} - \sqrt{s_2})^2.
\end{align*}
Thus, the condition $C(\pi,S)\geq 1$ is the same necessary and sufficient condition obtained by \citet{abbe2015exact} and  \citet{mossel2016} in the planted bisection model.}
\end{remark}

\begin{remark}
{\it In the symmetric SBM with balanced $k$ communities, that is when $\pi_i=1/k$, for $i=1,\dots,k$ and $S_{ij} = s_1$ for $i=j$ and $S_{ij} = s_2$ for  $i\neq j$ we have that
\begin{align*}
C(\pi,S)& =  \frac{1}k (\sqrt{s_1} - \sqrt{s_2})^2\,.
\end{align*}

In the general model with $k$ communities, the quantity $C(\pi,S)$ coincides with that obtained by 
\citet{abbe-colin2015}. They proved that exact recovery is information-theoretically solvable if, and only if, $C(\pi,S)\geq 1$. In this work, the strong consistency of the maximum likelihood estimator is not guaranteed at the phase transition threshold $C(\pi, S) = 1$, and whether consistency holds exactly at the threshold remains an open question for this estimator. 
}
\end{remark}

The next result provides a necessary condition that guarantees the strong consistency of the ICL estimator when $\rho_n \gg \log n/n$. This extends to the optimal sparse regime the consistency result for the Bayesian estimator established by \citet{van2018bayesian}, which was proved under the stronger assumption $\rho_n \gg \log^2 n/n$.

\begin{theorem}\label{teorema-chave-icl}
Let $(\Z,\A)$ be generated from a SBM with parameters $(\pi, \rho_nS)$ and $S$ is symmetric with no two columns equal. If  $\log n/n \ll \rho_n\leq 1$ we have that  
\begin{equation*}
\P\Bigl(\, L(\hz^\textsc{icl},\Z) \, = \,0\, \Bigr) \;\to\; 1
\end{equation*}
as $n\to\infty$. The result also holds at the regime $\rho_n=\log n/n$ and $C(\pi,S)\geq 1+k^2$.
\end{theorem}

 It is worth noting that, for the ICL estimator, we obtain a slightly stronger condition that depends on $k$, namely $C(\pi,S) \geq 1 + k^2$, in the regime $\rho_n= \log n/n$. Whether this condition can be relaxed for $\hz^\textsc{icl}$ remains an open question.

\section{Empirical analysis}\label{sec:empirical}

The ML and ICL estimators given in \eqref{zstar} and \eqref{zstar2}, respectively,  involve a maximization over all possible community assignments, making it computationally  demanding. In practice, it is possible to approximate the solution of the ML estimator by a solution of the variational EM algorithm, as proposed in \citet{daudin2008mixture}. In the case of the ICL estimator, \citet{latouche2015} proposed a greedy inference algorithm based on the integrated likelihood. While this algorithm differs slightly from the proposed ICL estimator in  \eqref{zstar2} by jointly estimating the number of communities and clusters, we utilize it as an illustration of the practical application of the ICL estimator. In order to measure the discrepancy between the estimated communities and the true ones, we opted to employ Normalized Mutual Information (NMI) because it is well-defined even for clusters with different numbers of communities, as can be the case in the computation of the ICL estimator \citep{Yao2003}. The NMI values range from 0 to 1, where 0 indicates no mutual information, and 1 indicates perfect agreement between the two communities' assignments. 

We first study the impact of the difference between the probability of connection within ($\rho_n s_1$) and between ($\rho_ns_2$) communities in terms of $(\sqrt{s_1}-\sqrt{s_2})^2$ for networks generated by the SBM model with balanced communities and $\rho_n=\log n / n $.   Figure \ref{fig:NMI_200} shows that the NMI increases to one as the quantity $(\sqrt{s_1}-\sqrt{s_2})^2$ approaches  $k$, visualized as the dashed vertical line, as stated in Theorem \ref{teorema-chave}. 
\begin{figure}[htb]
     \centering
     \begin{subfigure}[b]{0.45\textwidth}
         \centering
         \includegraphics[width=\textwidth]{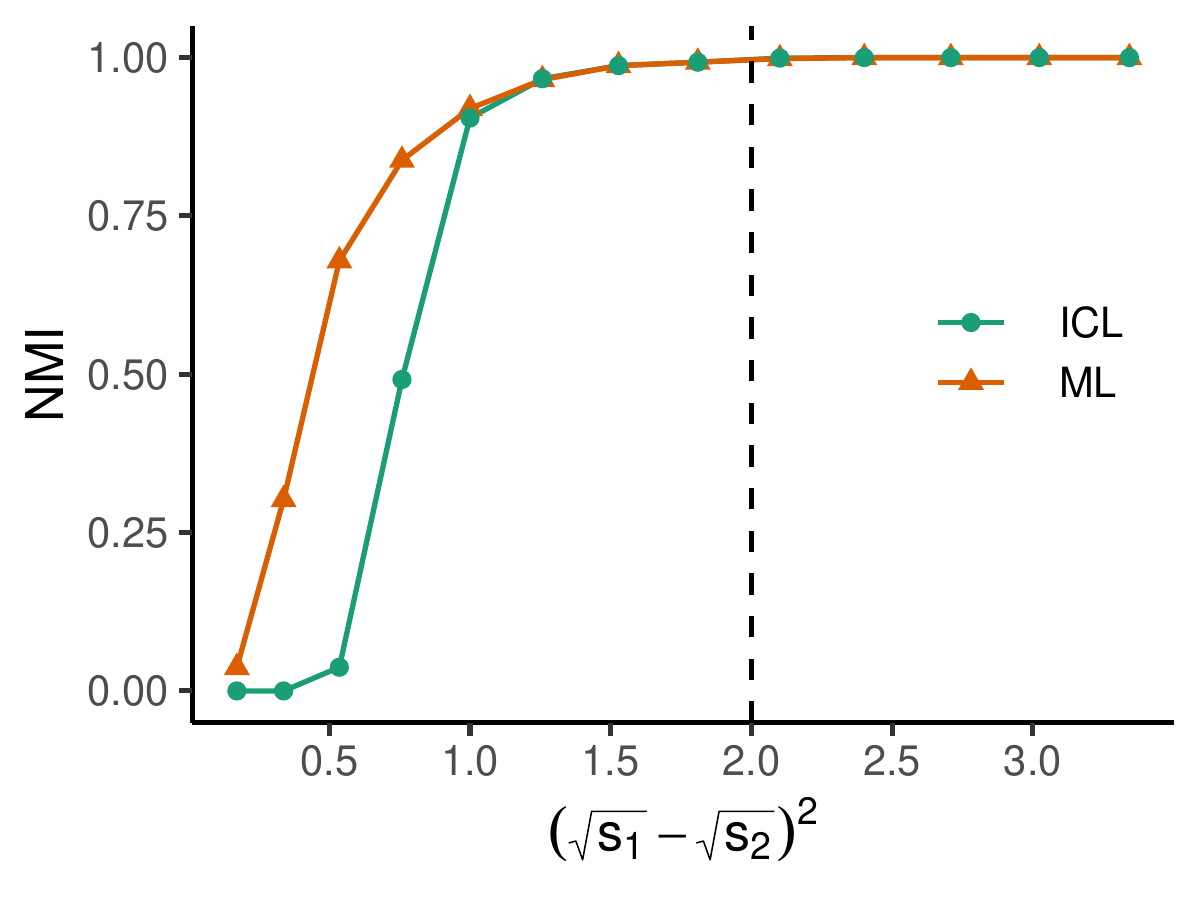}
         \caption{$k=2$.}
     \end{subfigure}
     \hfill
     \begin{subfigure}[b]{0.45\textwidth}
         \centering
         \includegraphics[width=\textwidth]{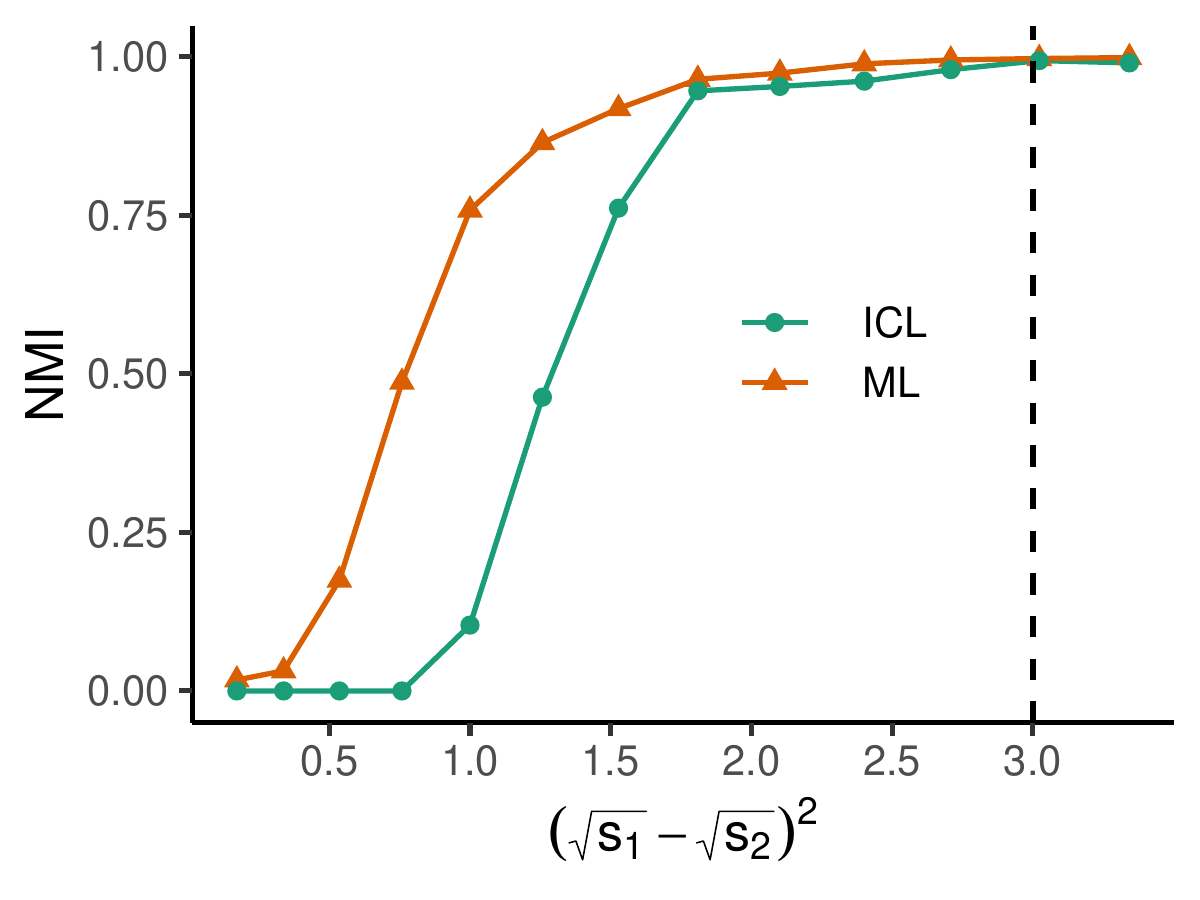}
         \caption{$k=3$.}
     \end{subfigure}
     \caption{Mean of NMI between estimated and true communities membership over 50 simulated balanced networks with $n=200$, $\rho_n=\log n/ n $ for the approximate solutions of ML and ICL estimators. The dashed vertical line shows the value of the constant  at which the phase transition occurs.}
\label{fig:NMI_200}
\end{figure}

To investigate the impact of the sparsity parameter $\rho_n$ in the NMI between the estimated and true communities assignments, we set $s_1$ and $s_2$ such that $(\sqrt{s_1}-\sqrt{s_2})^2 = 2.10 $ and we vary the values of $\rho_n$ for fixed $n=200$. Figure \ref{fig:NMI_rho} shows that setting $\rho_n=1/n=0.005$ the NMI is close to zero and it approaches to one when $\rho_n$ is getting closer to $\log n / n$.

\begin{figure}[htb]
\centering
     \begin{subfigure}[b]{0.45\textwidth}
         \centering
\includegraphics[width=\textwidth]{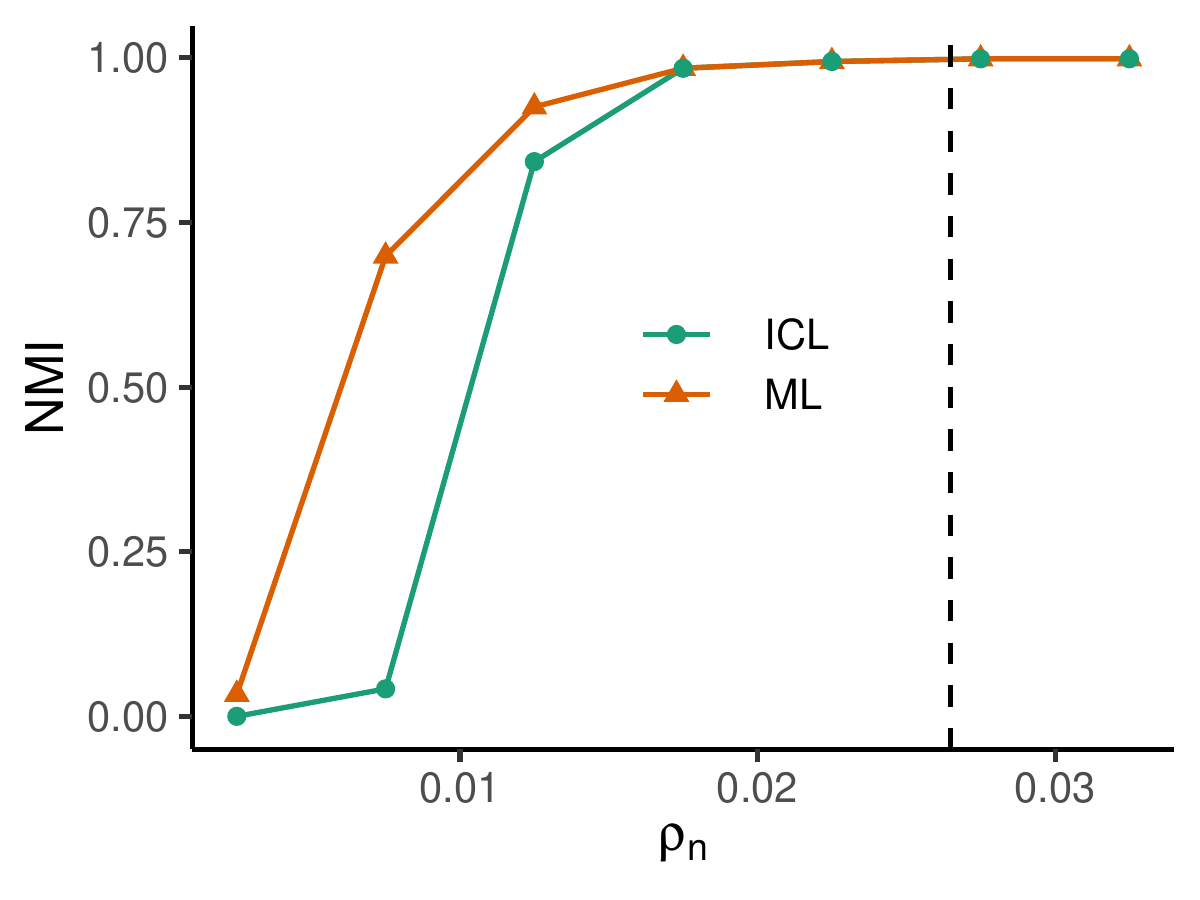}
         \caption{$k=2$.}
     \end{subfigure}
     \hfill
     \begin{subfigure}[b]{0.45\textwidth}
         \centering
         \includegraphics[width=\textwidth]{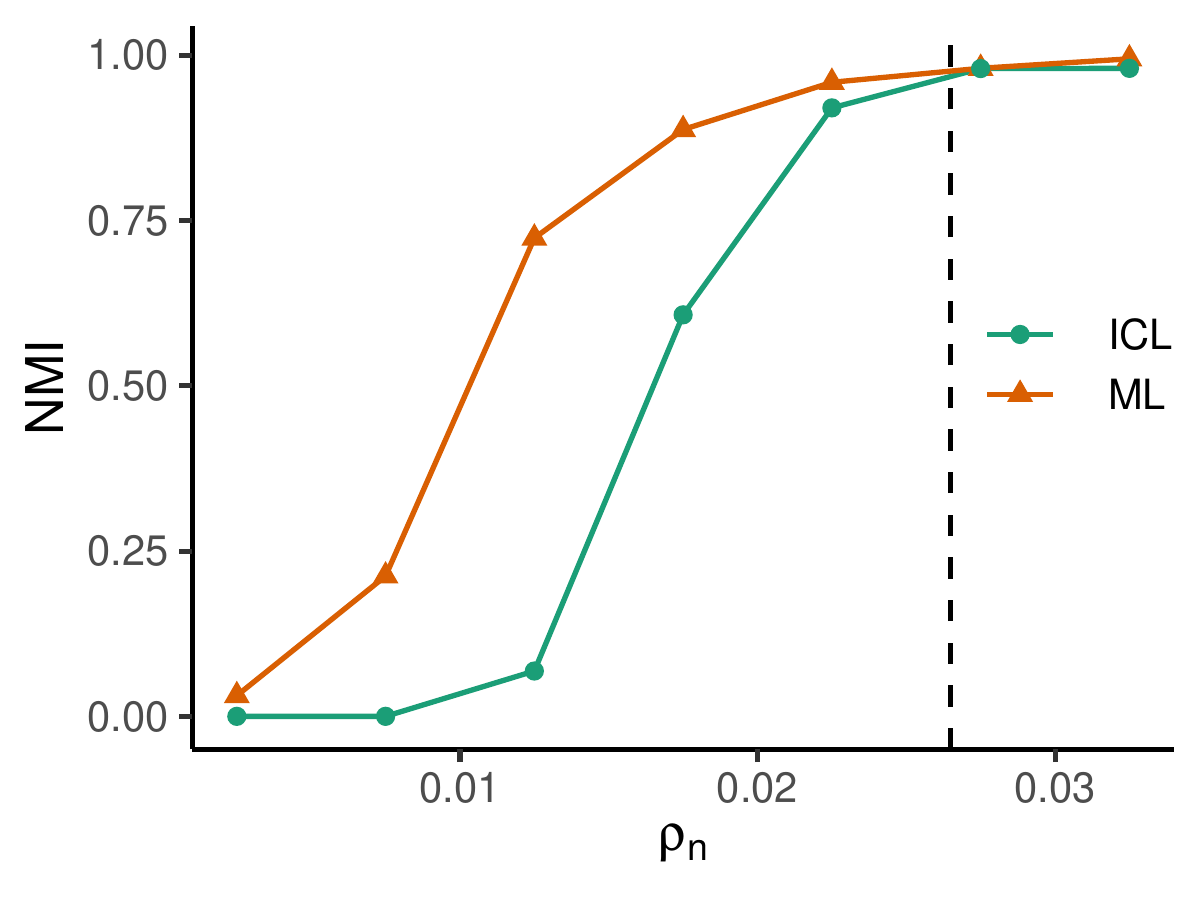}
         \caption{$k=3$.}
     \end{subfigure}
\caption{Mean of NMI between estimated and true communities membership over 50 simulated balanced networks with $n=200$, $k=2$, $(\sqrt{s_1}-\sqrt{s_2})^2>2$  for the approximation solution of the ML and ICL estimators. The dashed line is set at $\rho_n=\log n / n$.}
\label{fig:NMI_rho}
\end{figure}

\section{Proofs of main results}\label{proofs}

Let $\|M\|_1$ denote the entry-wise 1-norm of a matrix $M$, defined as the sum of the absolute values of all entries of $M$.

Define the discrepancy between two  community assignments $\e,\z \in [k]^n$ by
\[
m(\e,\z) \;=\;    \sum_{i=1}^n \1\{\hat e_i \neq  z_i\}. 
\] 
Observe that $m(\e,\z)$ coincides with $L(\e,\z)$ when the permutation of the labels $\sigma: [k]\to [k]$ is taken to be the identity.

Denote the confusion matrix between two  community assignments $\e,\z \in [k]^n$ by $R(\e,\z) \in \mathbb R_{\geq0}^{k\times k}$ with entries
\begin{equation}\label{conf_mat}
R_{ab}(\e,\z) \;=\;  {\frac1n}\sum_{i=1}^n \mathds{1}\{e_i=a, z_i=b\},\quad a,b\in [k]\,.
\end{equation}
Observe that the row sum and the column sum of the confusion matrix satisfy that $\na{a}(\e) = n [R(\e,\Z)\1]_a$ and $\na{a}(\z) = n [R(\e,\z)^T\1]_a$, respectively, where
 $\1$ is a $k\times 1$ vector with entries equal to one. Given a vector $v \in \mathbb R^k$ we denote by $\diag(v)$ the $k\times k$ diagonal matrix with diagonal 
elements given by $v$. An important observation is that $\diag(R(\e,\z)^T\1) = R(\z,\z)$, for all $\e,\z\in [k]^n$.

Given  any matrix $R\in \mathbb R_{\geq0}^{k\times k}$ and  $a,b\in[k]$ we define  the value 
\begin{equation}\label{pmixt}
 [P_R]_{ab} = \frac{[R P R^T]_{ab}-\delta_{ab}  \sum_{i=1}^{k}P_{ii} R_{ai}/n}{ [R\1]_{a}( [R\1]_{b}  - \delta_{ab}/n)}\,,
\end{equation}
where $\delta_{ab}=\1\{a=b\}$.

Observe that for  the model \eqref{def-prob}, for any  $\e,\z\in [k]^n$ we have that 
\begin{equation}\label{cond_exp_oab}
\E(\o{a}{b}(\e,\A)\,|\,\Z=\z) \;=\; \n{a}{b}(\e) [P_{R(\e,\z)}]_{ab}\,.
\end{equation}

Now define for any  confusion matrix  $R \in \mathbb R_{\geq0}^{k\times k}$  the functions
\begin{equation*}
\begin{split}
H_{P}(R) &\;=\; \frac12 \sum_{1\leq a,b\leq k} [R\1]_a[R\1]_b\, \tau\left( \frac{[R P R^T]_{ab}}{ [R\1]_{a}[R\1]_{b}} \right)
\end{split}
\end{equation*}
and
\begin{equation*}
\begin{split}
H_{P,n}(R) &\;=\; \frac12 \sum_{1\leq a,b\leq k} [R\1]_a([R\1]_b-\delta_{ab}/n)\, \tau\bigl( [ P_{R}]_{ab} \bigr)
\end{split}
\end{equation*}
where $P_R$ is defined in \eqref{pmixt} and $\tau(x)=x\log(x) + (1-x)\log(1-x)$. 
Observe that the function $H_{P,n}(R)$ is the likelihood modularity \eqref{def-P-hat} with the argument in the function $\tau$ substituted by its conditional expectation, given by \eqref{cond_exp_oab}. 

We also define, for two vectors $\e,\z \in [k]^n$,  the function $X\colon [k]^n\times [k]^n \to \mathbb R$  by 
\begin{equation}\label{Xmatrix}
X(\e,\z)\;=\; \frac1{2n^2}\sum_{1\leq a,b\leq k}\,\n{a}{b}(\e) \Bigl[ \tau\Bigl(\dfrac{\o{a}{b}(\e,\a)}{\n{a}{b}(\e)}\Bigr) - \tau\bigl( [ P_{R(\e,\z)}]_{ab} \bigr)\Bigr]\,.
\end{equation}
Note that, with this definition, the quantity $X(\e,\z)$ measures the difference in the likelihood when the arguments are given by the observed counts versus their expected values. Formally, we also have 
\begin{equation}\label{Xmatrix2}
X(\e,\z)\;=\;  \Q{\e} - H_{P,n}(R(\e,\z))\,.
\end{equation}

The proofs of weak and strong consistency for the ML and ICL estimators rely on concentration results for $X(\e,\z)$. 

\subsection{Proof of weak consistency}

We now turn to the proof of weak consistency for the ML and ICL estimators, as stated in Theorem~\ref{teorema-weak}.

\begin{proof}[Proof of Theorem~\ref{teorema-weak}]
It follows from Lemma~\ref{lemma-basico-vandpass} that
\[
m(\hz^\textsc{ml},\Z) \;=\; \frac{n}{2}\| R(\hz^\textsc{ml},\Z)-R(\Z,\Z)\|_1.
\]
Then, to prove that for all $\epsilon>0$, $L(\hz^\textsc{ml},\Z) < \epsilon n $ eventually almost surely as $n\to\infty$, it is enough to prove that
\begin{equation}\label{eq:dif_R}
\min_{\sigma}  \|D_\sigma R(\hz^\textsc{ml},\Z)-R(\Z,\Z)\|_1<\epsilon
\end{equation}
eventually almost surely as $n\to\infty$, where $D_\sigma$ denotes a permutation matrix associated with the permutation $\sigma:[k]\to [k]$ such that $D_\sigma R(\e,\Z)=R(\sigma(\e),\Z)$. \\
Let $\mathcal R_\epsilon$ be the set of probability matrices $R$ with
\[
\min_{\sigma} \|D_\sigma R - \diag(R^\intercal\1)\|_1\geq \epsilon\quad\text{ and }\quad \min_{a}( [R^\intercal\1]_a) \geq \epsilon\,.
\]
The set $\mathcal{R}_\epsilon$ consists of confusion matrices $R$ that cannot be transformed into a diagonal matrix by permuting its rows and contain a nonzero element in every column. \\
Observe that if $\epsilon$ is sufficiently small, then by the Law of Large Numbers we have that $ [R(\hz^\textsc{ml},\Z)^\intercal\1]_a \geq \epsilon$ simultaneously for all $a\in[k]$, eventually almost surely as $n\to\infty$. Then, $R(\hz^\textsc{ml},\Z)\not\in \mathcal R_\epsilon$ eventually almost surelly as $n\to\infty$ is equivalent to $L(\hz^\textsc{ml},\Z)\leq \epsilon n$ eventually almost surelly as $n\to\infty$, that is what we want to prove.

 We use an argument by contradiction, assuming that, for fixed and sufficiently small  $\epsilon>0$, \eqref{eq:dif_R} does not hold eventually almost surelly as $n\to\infty$, that is, the event  $\{ \|R(\hz^\textsc{ml},\Z)-R(\Z,\Z)\|_1\geq\epsilon\}$ has positive probability for infinitly many $n$.
 
Let
\[
\eta_n = \inf_{R\in \mathcal R_\epsilon} \{ \,H_{\rho_nS}(\diag(R^T\1))-H_{\rho_nS}(R)\,\}\,.
\]
Since  $R\in \mathcal R_\epsilon$ cannot be transformed into a diagonal element by permuting rows and $\mathcal R_\epsilon$ is a compact set, we have, by  Lemma~\ref{lemma-basico-vandpass0}, that $\eta_n\geq c\rho_n$ for all  $n$ sufficiently large and for a given constant $c>0$. Then 
\begin{equation}
\{ H_{\rho_nS}(\diag(R(\hz^\textsc{ml},\Z)^T\1))-H_{\rho_nS}(R(\hz^\textsc{ml},\Z)) \;\geq\; c\rho_n \}
\end{equation}
holds with positive probability for infinitely many $n$.  Observe that taking $\alpha \gg (n\rho_n)^{-1/2}$, we have that $[R(\hz^\textsc{ml},\Z)\1]_a \geq (n\rho_n)^{-1} $ for $n$ sufficiently large. Then, by Lemma~\ref{lem:hps} we also have that the event 
\begin{equation}\label{qml1}
\{ H_{\rho_nS,n}(\diag(R(\hz^\textsc{ml},\Z)^T\1))-H_{\rho_nS,n}(R(\hz^\textsc{ml},\Z)) \;\geq\; \frac{c}2\rho_n\}
\end{equation}
holds with positive probability for infinitely many $n$.
By the definition of $\hz^\textsc{ml}$, we have that 
\begin{equation}\label{qml2}
Q_{\textsc{ml}}(\hz^\textsc{ml})  - Q_{\textsc{ml}}(\Z)   \;\geq\; 0
\end{equation}
eventually almost surely as $n\to\infty$, provided $\alpha< \min_{a}\pi_a$. 
Therefore 
\[
 \{ X(\hz^\textsc{ml},\Z) -  X(\Z,\Z)  \;\geq \;  \frac{c}2\rho_n\}
\]
holds with positive probability for infinitely many $n$.
Taking $\alpha \gg (n\rho_n)^{-1/2}$, Theorem~\ref{theoremX} implies that for all $\delta>0$
\[
 |X(\hz^\textsc{ml},\Z) - X(\Z,\Z)|\;<\;  \delta\rho_n
\]
eventually almost surely as $n\to\infty$, what is a contradiction for $\delta<c/2$.
Then $R(\hz^\textsc{ml},\Z)\not\in \mathcal R_\epsilon$ eventually almost surelly as $n\to\infty$.
This concludes the proof of Theorem~\ref{teorema-weak} for the case of the maximum conditional likelihood estimator $\hz^\textsc{ml}$.\\
In the case of $\hz^\textsc{icl}$, by definition we obtain that
\[
Q_{\textsc{icl}}(\hz^\textsc{icl})  - Q_{\textsc{icl}}(\Z)   \;\geq\;  0 \,,
\] 
provided $\alpha< \min_{a}\pi_a$. 
Then, if \eqref{qml1} holds, by Lemma~\ref{lemma-qml-qb}  we have that 
\begin{equation*}
\begin{split}
 X(\hz^\textsc{icl},\Z) - X(\Z,\Z) 
&\;= \; H_{\rho_nS,n}(\diag(R^T\1))-  Q_{\textsc{icl}}(\Z)  + Q_{\textsc{icl}}(\Z) \\
&\quad\; - Q_{\textsc{ml}}(\Z)  -H_{\rho_nS,n}(R) + Q_{\textsc{icl}}(\hz^\textsc{icl}) -   Q_{\textsc{icl}}(\hz^\textsc{icl}) \\
&\quad\;+ Q_{\textsc{ml}}(\hz^\textsc{icl}) \\
&   \;\geq \; \frac{c}2\rho_n  - \frac{k^2 (\log n+2)}{n^2}\\
& \;\geq \; \frac{c}4\rho_n
\end{split}
\end{equation*}
with positive probability for infinitly many $n$. But we also have by  Theorem~\ref{theoremX} that 
\[
 |X(\hz^\textsc{icl},\Z) - X(\Z,\Z)|\;<\;  \delta\rho_n
\]
eventually almost surely as $n\to\infty$, what is a contradiction for $\delta<c/4$.  
\end{proof}

\subsection{Proof of strong consistency}

Now that we have established the weak consistency of the ML and ICL estimators, it follows that $L(\hz, \Z) < \epsilon n$  with probability converging to 1 as $n \to \infty$. To prove strong consistency, we need to show that the probability that the error of these estimators lies in the interval $[1, n]$ goes to zero as $n \to \infty$.

\begin{proof}[Proof of Theorem~\ref{teorema-chave}]
Let $\hz=\hz^\textsc{ml}$ and denote by $m=m(\hz,\Z)$ and $R = R(\hz,\Z)$. 
By Theorem~\ref{teorema-weak}, $\hz$ is weakly consistent, and therefore the confusion matrix $R$ is close to $\operatorname{diag}(R^\top \mathbf{1})$ up to an appropriate permutation of the labels. Without loss of generality, we assume that this permutation is the identity. Thus, by Lemma~\ref{lemma-basico-vandpass} and  Theorem~\ref{teorema-weak}, for $\delta>0$ 
\[
m\;=\;  \frac{n}{2}\,\|\diag(R^\top\1) - R\|_1  \;\leq \; \delta n
\]
holds with increasing probability. 
By the Strong Law of Large Numbers we also have that 
\[
\|\diag(R^\top\1) - \diag(\pi)\|_1 \;\leq \; \delta 
\]
eventually almost surely as $n\to\infty$, therefore
\[
\|R  - \diag(\pi)\|_1 \;\leq \; 2\delta
\]
eventually almost surely as $n\to\infty$.
If $\delta$ is taken sufficiently small, by Lemma~\ref{lemma-rate-nosso}  we have that 
\begin{equation}\label{mainineq}
\begin{split}
H_{\rho_n S,n}(\diag(R^\top\1)) - H_{\rho_nS,n}(R) \;&\geq \;  \rho_n \widetilde C(R,S)
\end{split}
\end{equation}
with 
\[
\widetilde C(R, S) =\sum_{b\neq b'} \Bigl( \sum_{a} [R^T\1]_{a} K_1(S_{ab}\|S_{ab'}) - \epsilon \Bigr)R_{b'b} \,,
\]
for $\epsilon >0$ arbitrarily small and $K_1(p\|q) = p \log (\frac{p}{q}) +  q - p$. 
Moreover,  we have that
\begin{equation}\label{eq378}
Q_{\textsc{ml}}(\hz)  - Q_{\textsc{ml}}(\Z)   \;\geq\; 0
\end{equation}
eventually almost surely as $n\to\infty$, provided $\alpha< \min_{a}\pi_a$. Then, by summing 
\eqref{mainineq} and \eqref{eq378} and  by \eqref{Xmatrix2}
 we must have that 
\begin{equation}\label{eq238}
 X(\hz,\Z)  - X(\Z,\Z) \;\geq \; \rho_n  \widetilde C(R,S) 
 \end{equation}
 eventually almost surely as $n\to\infty$.
We now define, for an arbitrary small $\eta>0$,  the event 
\[
\mathcal B = \Bigl\{ \sup_{a\in[k]}\Bigl| \frac{n_a(\Z)}n - \pi_a \Bigr| < \eta \Bigr\} . 
\]
We know, by the Strong Law of Large Numbers, that $\mathcal B$ holds, for any $\eta>0$,  eventually almost surely as $n\to\infty$. Thus, in the remaining of the proof, we will assume that $\Z \in \mathcal{B}$. On this set, we will prove that  
\begin{equation*}
\P\Bigl(\,  \bigl\lbrace  X(\hz,\Z) - X(\Z,\Z) \geq  \rho_n\widetilde C(R,S)\bigr\rbrace  \cap \bigl\lbrace m(\hz,\Z)\in (0,n]\bigr\rbrace \Bigr) \;\longrightarrow\; 0
\end{equation*}
as $n\to\infty$, implying that $m(\hz,\Z)=0$ with probability converging to 1 as $n\to\infty$.
Let $\Z \in \mathcal B$ and let $\e\in \F$ with $m(\Z,\e)\leq \delta n$. Conditioning on the event $\Z = \z$ and applying a union bound over the possible values of $m$ and $\e$, we obtain 
\begin{equation}\label{main-display0}
\begin{split}
&\P\Bigl(\bigl\lbrace  X(\hz,\Z) - X(\Z,\Z) \geq  \rho_n\widetilde C(R,S)\bigr\rbrace  \cap \bigl\lbrace m(\hz,\Z)\in (0,\delta n)\bigr\rbrace \Bigr) \\
&\leq \sum_{\z}\P\Bigl(\sup_{m\in [1,\delta n)} \sup_{\e\colon m(\e,\z)=m}X(\e,\z) - X(\z,\z) >  \rho_n \widetilde C(R,S) \,|\,\Z=\z\Bigr)\\ &\hspace{1.2cm}\times\P\bigl(\Z=\z\bigr)\\
&\leq \sum_{\z}\sum_{m}\sum_{\e} \P\Bigl( \{X(\e,\z) - X(\z,\z) >  \rho_n \widetilde C(R,S)\}\cap \mathcal C \,|\,\Z=\z\Bigr)\\
&\hspace{2.5cm}\times  \P\bigl(\Z=\z\bigr) + \P(\mathcal C^c)\,,
\end{split}
\end{equation}
where $\mathcal C$ is the event defined in \eqref{eq:event_C}. 
Now, if $\delta$ is sufficiently small, for any $\e$ with $m(\z,\e)\leq \delta n$,  by Theorem~\ref{theoremX2} we have that
\begin{equation}\label{eq:prob_diff_X}
\begin{split}
\P\bigl( \bigl\{X(\e,\Z)& - X(\Z,\Z)  >  \rho_n \widetilde C(R,S)\bigr\}\cap \mathcal C \,|\,\Z=\z\bigr)\\
&\;\leq\; \exp\Bigl[ - \sup_{t\in(0,1]} \Bigl\{  n^2\rho_n \bigl(  t \widetilde C(R, S) -  C_t(R,S)\bigr)\Bigr\} + \epsilon'\rho_n mn \Bigr]
\end{split}
\end{equation}
with $C_t(R,S)$  given by \eqref{tildeC}, and $\epsilon'$ a sufficiently small positive constant. 
Observe that for $t\in[0,1]$ we can write
\begin{align*}
t\widetilde C(R,S)-C_t(R, S) &=  \sum_{b\neq b'} \sum_{a} [R^T\1]_{a} \Bigl(tK_1(S_{ab}\|S_{ab'})- K_t(S_{ab}\|S_{ab'}) -  t\epsilon \Bigr)R_{b'b}\\
&\geq\;  \sum_{b\neq b'}\sum_{a}[R^T\1]_{a}  H_t(S_{ab} \| S_{ab'})R_{b'b}   - \frac{\epsilon m}{n} \\
&\geq\;  \Bigl( \min_{b\neq b'} \sum_{a}[R^T\1]_{a}  H_t(S_{ab} \| S_{ab'}) - \epsilon\Bigr)  \frac{m}{n}
\end{align*}
with
\[
H_t(s\| r) \;=\;  (1-t) s + t r - s^{1-t}r^t\,. 
\]
As $\Z\in \mathcal B$ and $H_t$ is bounded above,  for a sufficiently small $\eta$, we obtain that 
\begin{align*}
\sup_{t\in (0,1]}t\widetilde C(R,S)-C_t(R, S) &\geq\;  \bigl( C(\pi,S) -\epsilon\bigr) \frac{m}n 
\end{align*}
with
\[
C(\pi, S) \;=\; \min_{b\neq b'} \max_{t\in [0,1]} \;\sum_{a} \pi_a H_t(S_{ab} \| S_{ab'})
\]
and for a slightly different arbitrarily small $\epsilon>0$.
 Observe that $C(\pi,S) >0$ when $S$ is identifiable. Therefore
\begin{align*}
\P( \bigl\{X(\e,\Z) - X(\Z,\Z)  >   \rho_n \widetilde C(R,S)\bigr\}&\cap \mathcal C \,|\,\Z=\z)\\
&\;\leq\; \exp\Bigl[ - \bigl( C(\pi,S) -\epsilon\bigr) \rho_n n m  \Bigr],
\end{align*}
 with $\epsilon$ arbitrarily small. We use this bound on \eqref{main-display0} and count the number of $\e$ with $m(\z,\e)=m$, that can be bounded above by 
\[
 \binom{n}{m}(k-1)^m \;\leq \; \biggl(\frac{en(k-1)}{m} \biggr)^m
\]
to finally obtain that
\begin{equation}\label{main-display}
\begin{split}
\P\Bigl(\bigl\lbrace & X(\hz,\Z) - X(\Z,\Z) \geq  \rho_n\widetilde C(R,S)\bigr\rbrace  \cap \bigl\lbrace m(\hz,\Z)\in (0,\delta n)\bigr\rbrace \Bigr) \\
&\leq\;  \sum_{m=1}^{\delta n} \biggl(\frac{en(k-1)}{m} \biggr)^m\exp\Bigl[ - \bigl( C(\pi,S) -\epsilon\bigr) \rho_n n m  \Bigr]  + \P(\mathcal C^c)\\
&\leq\;  \sum_{m=1}^{\delta n}  \exp\bigl(  -  m  \bigl[ (C(\pi,S) - \epsilon)\rho_n n -  \log e(k-1)n \bigr]\bigr) +  \P(\mathcal C^c)\\
&\leq\;  \sum_{m=1}^{\delta n}  \exp\bigl(  -  m  \bigl[ (C(\pi,S) - \epsilon)\rho_n n - \log n-  \log e(k-1) \bigr]\big) +  \P(\mathcal C^c)\,.
\end{split}
\end{equation}
Now, for  $\rho_n \gg \log n/n$ and  $C(\pi,S)>0$ or  $\rho_n = \log n/n$ and  $C(\pi,S)>1$, and by Theorem~\ref{teo_basico1} that implies $\P(\mathcal C^c)\to 0$, the last expression in \eqref{main-display} converges to 0 as $n\to\infty$.

For the dense case ($\rho_n=1$), we apply Theorem \ref{theoremX2_denso} to replace $C_t$ in \eqref{eq:prob_diff_X} by $C'_t$ as defined in \eqref{tildeC_denso} to obtain
\begin{equation}\label{eq:prob_diff_X_denso}
\begin{split}
\P( \{X(\e,\z) - &X(\z,\z)  >    \widetilde C(R,S)\}\cap \mathcal C \mid\Z=\z)\\
&\;\leq\; \exp\Bigl[ - \sup_{t\in(0,1]} \Bigl\{  n^2\bigl(  t \widetilde C(R, S) -  C'_t(R,S)\bigr)\Bigr\} + \epsilon'mn \Bigr]
\end{split}
\end{equation}
with $\widetilde C(R,S)$ given by Lemma \ref{lemma-rate-nosso} by 
\begin{equation}
\tilde C(R,S)\geq \sum_{b\neq b'}  \Bigl(  \sum_{a} [R^T\1]_{a} KL(P_{ab}\| P_{ab'}) - \epsilon\Bigr)R_{b'b}.
\end{equation}
We can write
\begin{equation}
\begin{split}
tKL(p\|q)- D_t(p\|q) &= \log\left( \left(\frac{1-p}{1-q}\right)^t\right) - \log\left(p\left(\frac{q(1-p)}{p(1-q)} \right)^{t} +1-p\right)
\end{split}.
\end{equation}
First,  observe that  $tKL(p\|q)- D_t(p\|q)=0$ if $p=q$. Since $S$ has no two identical columns, this case does not occur. Thus, we want to show that $tKL(p\|q)- D_t(p\|q)= C >0$. To do so, it is enough to prove that
\begin{equation}
\left(\frac{1-p}{1-q}\right)^t > p\left(\frac{q(1-p)}{p(1-q)} \right)^{t} +1-p.
\end{equation}
In fact it holds since the function $x^t$, $t\in(0,1)$, is concave implying that $p(q/p)^t + (1-p)((1-q)/(1-p))^t > 1$. Thus,
\begin{align*}
\sup_{t\in (0,1]}t\widetilde C(R,S)-C_t(R, S) &\geq\;  \bigl( C -\epsilon\bigr) \frac{m}n .
\end{align*}
We proceeded as before to obtain that 
\begin{equation}
\begin{split}
\P\Bigl(\bigl\lbrace & X(\hz,\Z) - X(\Z,\Z) \geq  \widetilde C(R,S)\bigr\rbrace  \cap \bigl\lbrace m(\hz,\Z)\in (0,\delta n)\bigr\rbrace \Bigr) \\
&\leq\;  \sum_{m=1}^{\delta n}  \exp\bigl(  -  m  \bigl[ (C - \epsilon)n - \log n -  \log e(k-1) \bigr]\big) + \P(\mathcal C^c)\,.
\end{split}
\end{equation}
The sum converges to 0 as $n\to \infty$ since $C>0$, and the result follows for the dense case. This concludes the proof of Theorem~\ref{teorema-chave}.

\end{proof}

To prove the strong consistency of the ICL estimator, we use the relation between the likelihood and the integrated conditional likelihood modularities.  
Next, we establish the strong consistency of the ICL estimator defined in \eqref{zstar2}.

\begin{proof}[Proof of Theorem~\ref{teorema-chave-icl}]
In the case of $\hz=\hz^\textsc{icl}$, the proof is analogous to that of Theorem~\ref{teorema-chave} with some minor modifications. Observe that, as 
\[
Q_{\textsc{icl}}(\hz^\textsc{icl})  - Q_{\textsc{icl}}(\Z)   \;\geq\;  0 
\] 
now by Lemma~\ref{lemma-qml-qb} we have, instead of  \eqref{eq238} that  
\begin{equation}\label{eq232}
\begin{split}
X(\hz,\Z) - X(\Z,\Z) \;&>\;   \rho_n\widetilde C(R,S)  - \frac{ (k^2+1) \log n}{n^2}\,.
\end{split}
\end{equation}
 Then, as in \eqref{eq:prob_diff_X} we obtain that
\begin{align*}
\P( & \bigl\{X(\e,\z) - X(\z,\z) >   \rho_n \widetilde C(R,S) - \frac{ (k^2+1) \log n}{n^2}\bigr\}\cap \mathcal C \,|\,\Z=\z)\\
&\;\leq \; \exp\Bigl[ - \sup_{t\in(0,1]} \Bigl\{  n^2\rho_n \bigl(  t \widetilde C(R, S) -  C_t(R,S)\bigr)\Bigr\} +  (k^2+1) \log n + \epsilon'\rho_n mn \Bigr]\,.
\end{align*}
As before, the sum in the last display of  \eqref{main-display} becomes  
\begin{align*}
\sum_{m=1}^{\delta n} \exp\Bigl( -  m \Bigl[ \bigl(C(\pi,S) - \epsilon \bigr) \rho_n n -(2+k^2)\log n  -  \log e(k-1) \Bigr]\Bigr) + \P(\mathcal C^c)
\end{align*}
so, to obtain the strong consistency of $\hz^\textsc{icl}$ we need a little stronger condition, namely $C(\pi,S)>0$  for $\rho_n\gg \log n/n$ or $C(\pi,S)> 2+k^2$ for $\rho_n= \log n/n$.
 For the dense case ($\rho_n=1$), the last display of \eqref{main-display} becomes  
\begin{align*}
\sum_{m=1}^{\delta n} \exp\Bigl\{ -  m \Bigl[ \bigl(C - \epsilon \bigr)  n -(2+k^2)\log n  -  \log e(k-1) \Bigr]\Bigr\}.
\end{align*}
Since $C>0$, the result follows.
\end{proof}

\section{Discussion}\label{discussion}

In this paper, we proved the optimality of the maximum conditional likelihood estimator of the communities in a general stochastic block model with unknown parameters. That is, we derived the strong consistency of the ML estimator at the sparsity regime $\rho_n= \log n/n$ and with a constant $C(\pi,S)$ above the phase transition threshold for strong consistency, something that remained open in the literature until now. We also show that the integrated conditional likelihood is consistent in the same optimal sparse regime $\rho_n= \log n/n$, although with a stronger condition on the constant $C(\pi,S)\geq 1+k^2$. It remains open to show if this condition can be relaxed for this estimator. 
Considering almost exact recovery, we prove that a sufficient condition is $n\rho_n\to\infty$ as $n\to\infty$, 
generalizing the results obtained in \citet{mossel2016} that proved the sufficiency and necessity of this 
condition in the planted bisection model with known parameters,  and of  \citet{van2018bayesian}, that showed the weak consistenct of ML and Bayesian estimators on the regime $\rho_n= \log^2 n/n$. These results fill in the gap in the theory of community detection, showing that maximum and integrated conditional likelihood are optimal for community detection.  

As it is usual in the literature, in this paper, we assume a known number of communities. But the results could be extended to a number of communities growing with $n$, as the parameter $\alpha$ controlling the minimum size of a community can decrease with the sample size, at a rate $\alpha \gg  (n\rho_n)^{-1/2}$. This extension to a set of growing $k$ may need some additional assumptions, and it is out of the scope of this work. But it would allow the study of some estimators for the number of communities based on maximum likelihood, such as the classical Bayesian Information Criterion (BIC). The estimation of $k$ was addressed in  \citet{cerqueira20estimating}, who proposed a Bayesian estimator, and proved the consistency under the regime of weak consistency  $n\rho_n\to\infty$.  It remains an open question to know if the estimation of $k$ can be achieved with the BIC for the regime $n\rho_n\to\infty$. 

The estimation of $k$ is a key question in many statistical models, and in particular is challenging in latent variable models.  This problem is known as statistical model selection and has been addressed by Antonio Galves and co-authors in several of their recent works. An incomplete list includes the articles \citet{galves-leonardi2008, galvesetal2012, galves2013, duarteetal2019} and the book  \citet{galvesbook}.    

\section*{Acknowledgments}
This work was produced as part of the activities of the \emph{Research, Innovation and Dissemination Center 
for Neuromathematics} (FAPESP 2013/07699-0). It was also supported by FAPESP projects (grants 2017/10555-0 and 2023/05857-9) and CNPq  project (grant 441884/2023-7).  FL is partially supported by a CNPq’s research fellowship, grant 311763/2020-0.

\appendix

\renewcommand{\thetheorem}{\Alph{section}.\arabic{theorem}}

\section{Concentration results for weak consistency}

In this section we are interested in a concentration result for the empirical probabilities given by  
\[
\widehat P_{ab}(\A|\hz) \;=\; \frac{\o{a}{b}(\hz,\A)}{\n{a}{b}(\hz)}\,,\qquad 1\leq a,b\leq k\,,
\]
where $\hz$ is any of the estimators defined by \eqref{zstar} or \eqref{zstar2}. As no specific properties of $\hz$ are known, we control these probabilities uniformly over all possible vectors $\e\in [k]^n$. 

For simplicity in the notation, from now on we write $R(\e,\Z)=R(\e)$, for all $\e\in[k]^n$. 

\begin{theorem}\label{teo_basico1}
Let $(\Z,\A)$ be generated from a SBM  with parameters $(\pi, \rho_nS)\in \Theta^{k}$, with $\rho_n=1$  or $\rho_n\to0$  with $n\rho_n\to\infty$ as $n\to\infty$. Then, we have, for all $\delta>0$,  that 
\begin{equation}
 \Bigl|\widehat P_{ab}(\A\mid \Z) - \rho_n S_{ab}\Bigr| <   \frac{\sqrt{\delta \rho_n \log\log n}}{n}
\end{equation}
simultaneously for all $a,b\in[k]$, with probability converging to one as $n\to\infty$. Moreover, 
\begin{equation}
 \Bigl|\widehat P_{ab}(\A\mid \e) - [P_{R(\e)}]_{ab} \Bigr| <  \sqrt{\frac{8\sma \rho_n n \log k}{\n{a}{b}(\e)}}
\end{equation}
simultaneously for all $a,b\in[k]$ and all $\e\in
\F$, provided $\alpha \gg (n\rho_n)^{-1/2}$, eventually almost surely as $n\to\infty$.
\end{theorem}

\begin{proof}
For any $\e\in [k]^n$ and $a,b\in [k]$  consider the event 
\[
B^{ab}_{n}(\e)  = \Bigl\{  \o{a}{b}(\A\mid,\e) - \n{a}{b} [P_{R(\e)}]_{ab}  > x \Bigr\}.
\] 
Conditioning on $\Z=\z$,  we have that
\[
\o{a}{b}(\e,\A) -  \n{a}{b}(\e) [P_{R(\e)}]_{ab} = \sum_{1\leq i,j \leq n}\1\{e_i=a,e_j=b\}(A_{ij} - P_{z_i,z_j})
\]
is a sum of $\n{a}{b}(\e)$ zero mean independent random variables.
Applying Chernoff's bound we have that 
\begin{equation}
\begin{split}
\P\Bigl(B_n^{ab}(\e)  \mid \Z=\z \Bigr) \leq \inf_{t>0} \exp(-xt) \prod_{1\leq i,j\leq n} \frac{\E\bigl( \exp[  tY_{ij}] \mid \Z=\z\bigr)}{\exp[t\E(Y_{ij}\mid \Z=\z)]}, 
\end{split}
\end{equation}
where $Y_{ij} = \1\{e_i=a,e_j=b\}A_{ij}$. A simple calculation gives
\[
\E(Y_{ij}\mid \Z=\z) = \rho_nS_{z_iz_j}\1\{e_i=a,e_j=b\}
\]
and
\begin{equation}
\begin{split}
\E(\exp[ t Y_{ij}] \mid\Z=\z) &= \rho_n S_{z_iz_j} \exp(t \1\{e_i=a,e_j=b\}) + 1 - \rho_n S_{z_iz_j}\\
&\leq \exp\bigl(\rho_n S_{z_iz_j}\bigl[\exp(t \1\{e_i=a,e_j=b\}) - 1\bigr] \bigr).
\end{split}
\end{equation}
Then 
\begin{equation}
\begin{split}
\P\Bigl(&B_n^{ab}(\e)  \mid \Z=\z \Bigr) =\inf_{t>0}  \exp\Bigl(-xt + \sum_{1\leq i,j\leq n} \rho_n S_{z_iz_j}L(t \1\{e_i=a,e_j=b\})\Bigr)
\end{split}
\end{equation}
 with $L(y) = e^y-y-1$. Now we use that $L(y)\leq y^2$,  for $y\in (0,1]$ and obtain that 
\begin{equation}
\P\Bigl( B_n^{ab}(\e) \mid \Z=\z \Bigr) \leq \inf_{ t\in(0,1]}  \exp\Bigl(-xt + t^2\rho_n \sma \n{a}{b}(\e)\Bigr).
\end{equation}
By differentiating with respect to $t$ we obtain that the minimum is attained at 
\begin{equation}\label{tminimo}
t = \frac{x}{2\rho_n\n{a}{b}(\e)\sma}.
\end{equation}
Observe that the same can be done to bound the probability of the event 
\[
\widetilde B^{ab}_{n}(\e)  = \Bigl\{  \o{a}{b}(\A\mid\e) - \n{a}{b} [P_{R(\e)}]_{ab}  <  -x \Bigr\},
\] 
therefore for any $\e$ we obtain that 
\begin{equation}\label{main_ineq_en}
\begin{split}
\P\Bigl(\,\sup_{a,b \in [k]} & \Bigl|\widehat P_{ab}(\A\mid \e) - [P_{R(\e)}]_{ab} \Bigr|>x \mid \Z=\z \Bigr) 
\leq 2k^2\exp\Bigl(- \frac{x^2\n{a}{b}(\e)}{4\rho_n \sma}\Bigr).
\end{split}
\end{equation}
To guarantee that $t$, given by \eqref{tminimo}, belongs to $(0,1]$ for $n$ large enough, we need to choose $x$ such that
\begin{equation}\label{eq:taxa_x}
x = o(\rho_n \n{a}{b}(\e)).
\end{equation}
First, we take $\e=\z$.  As $\cap_{a=1}^{k_0} \{n_a(\Z) \geq \pi_{\min}n/2\}$ holds with probability one for $n$ large enough, we take 
\[
x = \frac{\sqrt{\delta \rho_n \log\log n}}{n} 
\]
that satisfies \eqref{eq:taxa_x},  since $n\rho_n\to \infty$ as $n\to \infty$. Thus, we obtain  that 
\[
\sup_{a,b\in[k_0]}\;\Bigl|\widehat P_{ab}(\A\mid \Z) - \rho_nS_{ab}\Bigr|<  \frac{\sqrt{\delta \rho_n \log\log n}}{n}
\]
with probability converging to one as $n\to\infty$. This proves the first statement of Theorem~\ref{teo_basico1}. To prove the second statement, we need to control the probability \eqref{main_ineq_en} simultaneously  for all $\e\in\F$, with $\alpha \gg (n\rho_n)^{-1/2}$, so in this case we take 
\[
x = \sqrt{\frac{8\sma\rho_n n\log k}{\n{a}{b}(\e)}}.
\]
Observe that the condition $\alpha \gg (n\rho_n)^{-1/2}$ implies that $x=o(\rho_n)$ so  \eqref{eq:taxa_x} is satisfied for all $\e\in\F$,  therefore we obtain that 
\begin{equation}\label{borel-cantelli}
\begin{split}
\P\Bigl(\bigcup_{a,b \in [k]}\bigcup_{\e \in  \F} & \Bigl\{\Bigl|\widehat P_{ab}(\A\mid \e) - [P_{R(\e)}]_{ab} \Bigr|>x\Bigr\} \mid \Z=\z \Bigr)\\
&\leq 2\sum_{a,b} \sum_{\e\in \F} \exp\Bigl(- \frac{x^2\n{a}{b}(\e)}{4\rho_n \sma}\Bigr)\\
&\leq 2 k^2\exp\bigl(- n \log k \bigr)\,.
\end{split}
\end{equation}
Therefore 
\[
\Bigl|\widehat P_{ab}(\A\mid \e) - [P_{R(\e,\z)}]_{ab} \Bigr|<  \sqrt{\frac{8\sma \rho_n n \log k}{\n{a}{b}(\e)}}
\]
simultaneously for all $a,b \in[k]$ and $\e \in \F$, with probability one for $n$ large enough, by an application of the Borel Cantelli Lemma, as the bound in \eqref{borel-cantelli} is summable in $n$. As all the bounds are uniform on $\z$, they hold also when substituting $R(\e,\z)$ by $R(\e,\Z)$. 
\end{proof}

The next corollary of Theorem~\ref{teo_basico1} is a key technical result that allows us to approximate the function $\tau$ in \eqref{def-P-hat} by its  Taylor expansion and ensures that the argument belongs to a bounded interval, where $\tau$ has bounded derivatives, implying that it is locally Lipschitz.  To guarantee the local Lipschitz property, we need to assume that the community assignments are of linear size and the matrix $S$ has bounded entries. This result complements and corrects the work by \citet{bickel2009nonparametric}, that assumes that $\tau$ is a global Lipschitz function, what is not true in general, as discussed in \citet{van2018bayesian}. On the other hand, the use of the  locally Lipchitz property of $\tau$ and the restriction to a bounded interval allows us to obtain optimal results for weak consistency, under the condition $n\rho_n\to\infty$ as $n\to\infty$, instead of the condition $n\rho_n\gg\log^2 n$ obtained by  \citet{van2018bayesian}. 

\begin{corollary}\label{main-cor}
Under the same hypotheses of Theorem~\ref{teo_basico1}, for all $\delta>0$ we have that  
\[
\frac{\o{a}{b}(\e,\A)}{\rho_n\,\n{a}{b}(\e)} \;\in \; (\smi - \delta, \sma+\delta)
\]
simultaneously for all $a,b \in [k]$  and  $\e\in\F$ with  $\alpha\gg (n\rho_n)^{-1/2}$ , eventually almost surely as $n\to\infty$. 
\end{corollary}

\begin{proof}
By Theorem~\ref{teo_basico1} we have, for any $\delta>0$, that 
\begin{equation}\label{evento}
 \Bigl| \,\widehat P_{ab}(\A|\e) - [P_{R(\e)}]_{ab}\, \Bigr| \;<\; \sqrt{\frac{8\sma \rho_n n \log k}{\alpha^2n^2}} \; <\; \delta\rho_n\,,
 \end{equation}
simultaneously for all $a,b\in [k]$ and $\e\in\F$, with  $\alpha\gg (n\rho_n)^{-1/2}$,  eventually almost surely as $n\to\infty$. 
Moreover, as $[P_{R(\e)}]_{ab}$ is a convex combination of $P$, we have that 
$[P_{R(\e)}]_{ab} \in (\rho_n\smi, \rho_n\sma)$.
Then, the result follows. 
\end{proof}

Using Corollary~\ref{main-cor} and a Taylor expansion of the function $\tau$, for the sparse regime, we can derive the next result.

\begin{theorem}\label{theoremX}
Let $(\Z,\A)$ be generated from a SBM with parameters $(\pi, \rho_nS)$, with $\rho_n=1$ or $\rho_n\to0$ with $n\rho_n\to\infty$ as $n\to\infty$.
Then for all $\delta > 0$ we have that 
\begin{equation*}
|X(\e,\Z)- X(\Z,\Z)| \; <\; \delta\rho_n 
\end{equation*}
simultaneously for all $a,b\in [k]$ and $\e\in\F$, with  $\alpha\gg (n\rho_n)^{-1/2}$,  eventually almost surely as $n\to\infty$. 
\end{theorem}

\begin{proof}
For simplicity in the notation  we write $X(\e)=X(\e,\Z)$, $X(\Z)=X(\Z,\Z)$, $R(\e,\Z)=R(\e)$ and $R(\Z,\Z)=R(\Z)$. \\
Consider the sparse case with $\rho_n\to0$  with $n\rho_n\to\infty$ as $n\to\infty$. In this case, we use the Taylor expansion of the function $\tau$, given by 
\begin{equation*}
\tau(\rho_n x) \;=\; \rho_n \gamma(x) +  x \rho_n \log\rho_n + O\bigl(\rho_n^2x^2\bigr)
\end{equation*}
with
\[
\gamma(x) \;=\; x\log x - x\,.
\]
By Corollary~\ref{main-cor}, we obtain that
\begin{equation}\label{second_term}
\begin{split}
 \tau\Bigl(\dfrac{\o{a}{b}(\e,\A)}{\n{a}{b}(\e)}\Bigr) \;&=\;  
\rho_n \gamma\Bigl(\dfrac{\o{a}{b}(\e,\A)}{\rho_n\n{a}{b}(\e)}\Bigr) + \dfrac{\o{a}{b}(\e,\A)}{\rho_n\n{a}{b}(\e)} \rho_n\log \rho_n  + O(\rho_n^2)\\
\end{split}
\end{equation}
and a similar expression for $\dfrac{\o{a}{b}(\Z,\A)}{\n{a}{b}(\Z)}$, $[ P_{R(\e)}]_{ab}$ and $[ P_{R(\Z)}]_{ab}$. 
Therefore, we have that 
\begin{equation}\label{expressaoX}
\begin{split}
X(\e) - X(\Z) \;= \; 
&  \frac12\;  \Biggl( \; \sum_{a,b\in[k]}\,\frac{ \rho_n n_{ab}(\e)}{n^2} \Bigl[ \gamma\Bigl(\dfrac{\o{a}{b}(\e,\A)}{\rho_n \n{a}{b}(\e)}\Bigr) - \gamma\bigl( [ P_{R(\e)}]_{ab}/\rho_n \bigr)\Bigr] \\
&-  \sum_{a,b\in[k]}\,\frac{ \rho_n n_{ab}(\Z)}{n^2} \Bigl[ \gamma\Bigl(\dfrac{\o{a}{b}(\Z,\A)}{\rho_n\n{a}{b}(\Z)}\Bigr) - \gamma\bigl( [ P_{R(\Z)}]_{ab}/\rho_n \bigr)\Bigr]\Biggr)\\
& +  O(k^2\rho_n^2)
\end{split}
\end{equation}
as the second term in the right hand side of \eqref{second_term} cancels when making the difference $X(\e) - X(\Z)$, because 
\[
\sum_{a,b} \o{a}{b}(\e,\A) \;=\; \sum_{a,b} \o{a}{b}(\Z,\A)
\]
and
\[
\sum_{a,b} [R(\e)PR(\e)^T]_{ab}  \;=\; \sum_{a,b} [R(\Z)PR(\Z)^T]_{ab}\,.
\]
The function $\gamma$ has bounded derivatives in any bounded interval, and we know by 
Corollary~\ref{main-cor} that 
\[
\frac{\o{a}{b}(\e,\A)}{\rho_n\,\n{a}{b}(\e)} \;\in \; \Bigl[ \frac{\smi}2, \frac{3\sma}2\Bigr]
\]
simultaneously for all $a,b \in [k]$  and  $\e\in \F$, eventually almost surely as $n\to\infty$,  provided $\alpha \gg (n\rho_n)^{-1/2}$ for all $n$ large enough.  
Then, calling  $M$ the upper bound on the absolute value of the  derivative of $\gamma$ over this interval, we obtain that  
\begin{align*}
|X(\e)- X(\Z)| \;\leq\; &\;\frac12 M\Biggl( \sum_{1\leq a, b\leq k}\; \Bigl|\frac{\o{a}{b}(\e,\A)}{n^2} -[R(\e)PR(\e)^T]_{ab} \Bigr| \\
&+  \sum_{1\leq a, b\leq k}\; \Bigl|\frac{\o{a}{b}(\Z,\A)}{n^2} -[R(\Z)PR(\Z)^T]_{ab} \Bigr|\Biggr)\\
&+ O(k^2\rho_n^2)
\end{align*}
simultaneously for all $a,b \in [k]$  and  $\e\in \F$  with $\alpha \gg (n\rho_n)^{-1/2}$, eventually almost surely as $n\to\infty$. 
By  Theorem~\ref{teo_basico1} we obtain that 
\begin{align*}
|X(\e)- X(\Z)| \;&\leq\; \frac{M}{2} \sum_{1\leq a, b\leq k}\; \sqrt{\frac{8\sma\rho_n \n{a}{b}(\e) \log k}{n^3}} + \sqrt{\frac{8\sma\rho_n  \n{a}{b}(\Z)\log k}{n^3}} \\
&\quad+ O(k^2\rho_n^2)\\
&\leq\; k^2M \rho_n  \sqrt{\frac{8\sma \log k}{\rho_nn}} + O(k^2\rho_n^2)\\
&\leq \; \delta \rho_n
\end{align*}
simultaneously for all $a,b \in [k]$  and  $\e\in \F$  with $\alpha \gg (n\rho_n)^{-1/2}$, provided $\rho_n n\to\infty$ as $n\to\infty$, eventually almost surely.
This finishes the proof  of Theorem~\ref{theoremX}.
The dense case ($\rho_n = 1$) can be obtained directly by using the function $\tau$ in place of $\gamma$ in \eqref{expressaoX}, following the same arguments as for the sparse case.
\end{proof}

\section{Concentration results for strong consistency}

The bound on the difference $X(\hz)- X(\Z)$ given in  Theorem~\ref{theoremX} is sufficient to prove the weak consistency of $\hz^\textsc{ml}$ and 
$\hz^\textsc{icl}$, stated in Theorem~\ref{teorema-weak}. However, in order to prove the strong consistency, that is Theorem~\ref{teorema-chave},  we need a better control of the difference  $X(\e)- X(\Z)$, in terms of $m(\e,\Z)$, for any vector $\e\in [k]^n$ and of $\rho_n\to 0$, as $n\to \infty$.  This is done in Theorem~\ref{theoremX2} 
 and in the subsequent results in this section. To control this difference, we define, for any $a,b\in[k]$ and $\e,\z\in [k]^n$, the quantity
\begin{equation}\label{Wab}
W_{ab}(\e,\z) = \dfrac{\o{a}{b}(\e)}{n^2}  - \dfrac{\E(\o{a}{b}(\e))}{n^2}- \dfrac{\o{a}{b}(\z)}{n^2}+  \dfrac{\E(\o{a}{b}(\z))}{n^2}.
\end{equation} 
We also define the event
\begin{equation}\label{eq:event_C}
\mathcal C = \Bigl\{ \,\sup_{a,b\in[k]}\Bigl|\widehat P_{ab}(\A\mid \Z) - \rho_n S_{ab}\Bigr| <   \frac{\sqrt{ \rho_n \log \log n}}{n}
 \,\Bigr\}
\end{equation}
which holds with probability converging to one as $n\to\infty$, by Theorem~\ref{teo_basico1}.

In the remaining of this section, we will focus on  the sparse case, that is when $\rho_n\to0$ as $n\to\infty$. For the dense case, $\rho_n=1$, the complementary results are given in \ref{sec:dense}.

\begin{theorem}\label{theoremX2}
Let $(\Z,\A)$ be generated from a SBM with parameters $(\pi, \rho_nS)$, with $S>0$ fixed and $\rho_n\to 0$ with $\rho_n \geq  \log n/n$. Let $\z\in [k]^n$ be such that  $n_{a}(\z) \geq \pi_{\min}n/2$ for all $a\in[k]$ and $\e\in \F$. Then, for all $\epsilon>0$, there exists $\tilde\epsilon>0$ such that, if   $m=m(\z,\e)\leq \tilde\epsilon n$ and  
$n$ is sufficiently large, we have that
\begin{align*}
\P( \{ X(\e) - X(\Z)  >  x\}&\cap  \mathcal C\,|\,\Z=\z) \\
&\leq  \exp\Bigl[  - \sup_{t\in(0,1]} \Bigl\{  n^2\bigl(x t -  \rho_n C_t(R,S)\bigr)\Bigr\} + \epsilon\rho_n nm\Bigr]
\end{align*}
with $C_t(R,S)$ given by \eqref{tildeC}.
\end{theorem}

\begin{proof}
From now on, we will assume that \eqref{Waball} and \eqref{oes} in Corollary \ref{corrate} hold, as this happens simultaneously for all $a,b\in[k]$ and all $\e\in\F$ with $m(\z,\e)\leq \pi_{\min}^2n/16$ eventually almost surely as $n\to\infty$.    
Therefore, all subsequent computations are carried out under the event $\mathcal{C}$ and assuming the bounds in Corollary \ref{corrate}.
For simplicity of notation, we denote $\o{a}{b}(\e,\A)$ by $\o{a}{b}(\e)$, for any $\e\in[k]^n$. For $\z\in[k]^n$, we denote $R(\e,\z)=R(\e)$. Observe that 
\begin{equation*}
\begin{split}
X(\e)& - X(\z)\;= 
\; \frac12\Biggl(\sum_{1\leq a, b\leq k}\, \dfrac{\n{a}{b}(\e)}{n^2} \Bigl[\tau\Bigl(\dfrac{\o{a}{b}(\e)}{\n{a}{b}(\e)}\Bigr)  -  \tau\Bigl([ P_{R(\e)}]_{ab}\Bigr)  \Bigr]\\
&\qquad\qquad\qquad\qquad\qquad +  \dfrac{\n{a}{b}(\z)}{n^2}  \Bigl[\tau\Bigl([ P_{R(\z)}]_{ab}\Bigr)  - \tau\Bigl(\dfrac{\o{a}{b}(\z)}{ \n{a}{b}(\z)}\Bigr)  \Bigr]\Biggl)\,,
\end{split}
\end{equation*}
with 
\[
\tau(x) = x\log x + (1-x) \log(1-x)\,.
\]
By the Gibbs inequality, we have, in the case of $x,y\in (0,1)$ that 
\[
\tau(x) \geq x\log y + (1-x) \log(1-y)\,.
\]
We use this inequality for $x=[ P_{R(\e)}]_{ab}$ and $y=\dfrac{\o{a}{b}(\e)}{\n{a}{b}(\e)}$ and we obtain that
\[
 \tau\Bigl([ P_{R(\e)}]_{ab}\Bigr)  \;\geq\; [ P_{R(\e)}]_{ab}\log \dfrac{\o{a}{b}(\e)}{\n{a}{b}(\e)} + (1 - [ P_{R(\e)}]_{ab}\log \Bigl(1-\dfrac{\o{a}{b}(\e)}{\n{a}{b}(\e)}\Bigr)\,.
\]
We also use it for $x=\dfrac{\o{a}{b}(\z)}{\n{a}{b}(\z)}$ and $y= \rho_n S_{ab}$ and we obtain that
\[
 \tau\Bigl(\dfrac{\o{a}{b}(\z)}{\n{a}{b}(\z)}\Bigr)  \;\geq\; \dfrac{\o{a}{b}(\z)}{\n{a}{b}(\z)}\log\rho_n S_{ab} + (1 - \dfrac{\o{a}{b}(\z)}{\n{a}{b}(\z)})\log \Bigl(1-\rho_n S_{ab}\Bigr)
\]
From the definition in \eqref{pmixt}, it follows that
\[
[ P_{R(\z)}]_{ab} = \rho_nS_{ab}\,.
\]
Then we obtain the bound 
\begin{equation}\label{XX}
\begin{split}
X(\e) - X(\z)\;\leq\;&  \frac12\sum_{1\leq a, b\leq k}\, \Bigl(\dfrac{\o{a}{b}(\e)}{n^2}- 
\E\Bigl(\dfrac{\o{a}{b}(\e)}{n^2}\Bigr)
\Bigr) \log \Bigl(\dfrac{\o{a}{b}(\e)}{\n{a}{b}(\e)}\Bigr) \\
&+  \Bigl(\E\Bigl(\dfrac{\o{a}{b}(\e)}{n^2}\Bigr) - \frac{\o{a}{b}(\e)}{n^2}\Bigr) \log\Bigl(  1 -  \frac{\o{a}{b}(\e)}{\n{a}{b}(\e)}\Bigr)\\
&+ \Bigl(\E\Bigl(\dfrac{\o{a}{b}(\z)}{n^2}\Bigr)- \dfrac{\o{a}{b}(\z)}{n^2}\Bigr) \log \rho_nS_{ab}\\
&+ \Bigl(\dfrac{\o{a}{b}(\z)}{n^2} - \E\Bigl(\dfrac{\o{a}{b}(\z)}{n^2}\Bigr)\Bigr)\log ( 1- \rho_nS_{ab})\,,
\end{split}
\end{equation}
where, by an abuse of notation, we write $\E\Bigl(\dfrac{\o{a}{b}(\e)}{n^2}\Bigr) = \E\Bigl(\dfrac{\o{a}{b}(\e)}{n^2}\,|\,\Z=\z\Bigr)$. 
Observe that we can write 
\begin{align*}
\log \Bigl(\dfrac{\o{a}{b}(\e)}{\n{a}{b}(\e)}\Bigr) \;=\; \log \Bigl[ \rho_nS_{ab} \Bigl(1 + \frac{ \o{a}{b}(\e)/\n{a}{b}(\e) - \rho_nS_{ab}}{\rho_nS_{ab}}\Bigr)\Bigr]
\end{align*}
and similarly 
\begin{align*}
\log \Bigl(1 - \dfrac{\o{a}{b}(\e)}{\n{a}{b}(\e)}\Bigr) \;=\; \log \Bigl[ \bigl(1-\rho_nS_{ab}\bigr) \Bigl(1 + \frac{ \rho_nS_{ab} - \o{a}{b}(\e)/\n{a}{b}(\e)}{1-\rho_nS_{ab}}\Bigr)\Bigr].
\end{align*}
On the event $\mathcal C$, we have, by  Corollary~\ref{corrate}  and the assumptions $m(\z,\e)\leq \tilde\epsilon n$ and $\rho_n\geq \log n /n$,  that
\begin{equation}\label{toprove}
\Bigl| \frac{\o{a}{b}(\e)}{\n{a}{b}(\e)} - \rho_nS_{ab} \Bigr| \leq \Bigl| \frac{\o{a}{b}(\e)}{\n{a}{b}(\e)} -  \frac{\o{a}{b}(\z)}{\n{a}{b}(\z)} \Bigr|  + \Bigl| \frac{\o{a}{b}(\z)}{\n{a}{b}(\z)} - \rho_nS_{ab} \Bigr| < \epsilon\rho_n
\end{equation}
simultaneously for all $a,b\in[k]$, provided $\tilde\epsilon>0$ is sufficiently small and $n$ is sufficiently large.  Then by using the inequality $\log(1+x)\leq |x|$ valid  for $x > -1$, 
the difference $X(\e) - X(\z)$  in \eqref{XX} can be upper bounded by 
\begin{equation}\label{XXX}
\begin{split}
 \frac12\sum_{1\leq a, b\leq k}\, &W_{ab}(\e,\z) \log  \frac{\rho_n S_{ab}}{1-\rho_n S_{ab}}\\
& + \frac12\sum_{1\leq a, b\leq k} \frac{1}{ \rho_nS_{ab}} \; \Bigl|\E\Bigl(\dfrac{\o{a}{b}(\e)}{n^2}\Bigr) - \frac{\o{a}{b}(\e)}{n^2}\Bigr|\,
\Bigl| \dfrac{\o{a}{b}(\e)}{\n{a}{b}(\e)} - \rho_nS_{ab}\Bigr|\\
& + \frac12\sum_{1\leq a, b\leq k} \frac{1}{1- \rho_nS_{ab} } \; \Bigl|\E\Bigl(\dfrac{\o{a}{b}(\e)}{n^2}\Bigr) - \frac{\o{a}{b}(\e)}{n^2}\Bigr|\,
\Bigl| \dfrac{\o{a}{b}(\e)}{\n{a}{b}(\e)} - \rho_nS_{ab}\Bigr|\,,
\end{split}
\end{equation}
where $W_{ab}(\e,\z)$ was defined in \eqref{Wab}. 

From this point on we assume, without loss of generality, that  $\rho_n S_{ab}\leq 1-\rho_nS_{ab}$ for all $a,b$. Observe that this is true in the sparse case where $\rho_n\to 0$, for $n$ large enough. Then, we can bound above \eqref{XXX} by the sum of two terms 
$W(\e,\z) + \widetilde W(\e,\z)$ with 
\[
W(\e,\z) = \frac12 \sum_{1\leq a,b\leq k} W_{ab}(\e,\z)  \log  \frac{\rho_n S_{ab}}{1-\rho_n S_{ab}}
\]
and
\[
\widetilde W(\e,\z) = \sum_{1\leq a, b\leq k} \frac{1}{ \rho_nS_{ab}} \; \Bigl|\E\Bigl(\dfrac{\o{a}{b}(\e)}{n^2}\Bigr) - \frac{\o{a}{b}(\e)}{n^2}\Bigr|\,
\Bigl| \dfrac{\o{a}{b}(\e)}{\n{a}{b}(\e)} - \rho_nS_{ab}\Bigr|.
\]
We will now obtain a bound for $\widetilde W(\e,\z)$.
By \eqref{Waball} in Corollary~\ref{corrate} we have 
\begin{equation}\label{rateWab}
\max_{ab} \;  \Bigl|W_{ab}(\e,\z)\Bigr| \;\leq\; \frac{c\rho_n m(\e,\z)}{n}
\end{equation}
with $c$ a given constant only depending on $S$. 
Then on $\mathcal C$ we obtain the bound
\begin{equation}\label{bound1}
\begin{split}
 \Bigl|\E\Bigl(\dfrac{\o{a}{b}(\e)}{n^2}\Bigr) - \frac{\o{a}{b}(\e)}{n^2}\Bigr| &\;\leq\;   \Bigl|W_{ab}(\e,\z)\Bigr|  + \Bigl| \frac{\o{a}{b}(\z)}{n^2}- \E\Bigl(\dfrac{\o{a}{b}(\z)}{n^2}\Bigr)\Bigr| \\
 &\;\leq\; \frac{c\rho_n m(\e,\z)}{n} +  \frac{\sqrt{\rho_n \log \log n}}{n}
\end{split}  
 \end{equation}
simultaneously for all $a,b\in[k]$. Observe also that, by the triangle inequality, we can write 
\begin{equation}\label{bound2}
\begin{split}
\Bigl| \dfrac{\o{a}{b}(\e)}{\n{a}{b}(\e)} - \rho_nS_{ab}\Bigr| \;&\leq\; 
\Bigl| \dfrac{\o{a}{b}(\e)}{\n{a}{b}(\e)} -\dfrac{\o{a}{b}(\z)}{\n{a}{b}(\z)} \Bigr| + \Bigl| \dfrac{\o{a}{b}(\z)}{\n{a}{b}(\z)} - \rho_nS_{ab}\Bigr|.
\end{split}
\end{equation}
 Then, on $\mathcal C$ and by \eqref{oes} in Corollary~\ref{corrate} we also obtain the upper bound 
\begin{equation}\label{bound-2}
 \Bigl| \dfrac{\o{a}{b}(\e)}{\n{a}{b}(\e)} - \rho_nS_{ab}\Bigr| \;\leq\;  \frac{c\rho_n m(\e,\z)}{n} +  \frac{\sqrt{\rho_n \log\log n}}{n}. \end{equation}
Then, the bounds \eqref{bound1} and \eqref{bound-2} implies that
\[
\widetilde W(\e,\z) \leq  \frac{2k^2}{\rho_n\smi} \Bigl( \frac{c^2\rho_n^2m^2}{n^2} +  \frac{\rho_n \log\log n}{n^2}  \Bigr) \leq \frac{C (\rho_nm^2+ \log\log n)}{n^2}
\]
with $C$ a constant only depending on $S$ and $k$. 
By Proposition~\ref{lemma-chernoff}, we obtain that 
\begin{align*}
\P\Bigl( \{X(\e) - &X(\z)  > x\} \cap \mathcal C \mid \Z=\z \Bigr)\\
&\leq\; \P\Bigl( W(\e,\z) > x -  \widetilde W(\e,\z)  \mid \Z=\z  \Bigr)\\
 &\leq\; \exp\Bigl[ - \sup_{t\in (0,1]} \Bigl\{ n^2 (xt-  \rho_n C_t(R,S))\Bigr\} + \phi(n,m)\Bigr]\,,
\end{align*}
with 
\[
\phi(n,m) = C( \rho_n m^2+ \log\log n+ \rho_n^2 nm)
\]
and $C$ a constant only depending on $S$ and $k$. 
Assuming $m(\z,\e)\leq \tilde\epsilon n$, with $\tilde\epsilon$ sufficiently small,  and $n\rho_n\geq \log n$, we obtain that 
\[
 \phi(n,m) \leq \epsilon \rho_nnm 
\]
for $n$ sufficiently large, and the result follows.
\end{proof}

We conclude this section by proving the key results used in the proof of Theorem~\ref{theoremX2}, namely 
Corollary~\ref{corrate}, derived from Propositions~\ref{lemma-chernoff2},  and Proposition~\ref{lemma-chernoff}, giving the main result that derives in the optimality of the ML and ICL estimators. 

The following proposition is valid for the sparse and dense regimes. 

\begin{proposition} \label{lemma-chernoff2}
Let $(\Z,\A)$ be generated from a SBM with parameters $(\pi, \rho_nS)$,  with $\rho_n=1$ or $\rho_n\to 0$ as $n\to\infty$. 
Then for all $a,b\in[k]$ and $\e\in[k]^n$ we have that
\begin{equation*}
\begin{split}
 \P\Bigl(\,  & |W_{ab}(\e,\Z)|  \; >\; x\,|\,\Z=\z  \, \Bigr) \;\leq\; 2 \exp\Bigl(- xn^2 + 6\rho_n  \sma  m(\e,\z) n\Bigr)\,.
\end{split}
\end{equation*}
\end{proposition}

\begin{proof}
Let 
\[
Y^{ab} = \sum_{1\leq i<j\leq n} Y^{ab}_{ij}
\]
with  
\begin{align*}
Y_{ij}^{ab} = \frac{1}{n^2}&\Bigl( \mathds{1}\{e_i=a,e_j=b\}+ \mathds{1}\{e_i=b,e_j=a\} \\
&- \mathds{1}\{z_i=a,z_j=b\} - \mathds{1}\{z_i=b,z_j=a\} \Bigr)A_{ij}\,.
\end{align*}
Conditioned on $\Z=\z$, the variables $Y_{ij}^{ab}$, for $1\leq i<j\leq n$ are independent. 
By applying Chernoff's bound, we obtain that
\begin{align*}
\P( \,Y^{ab} - \E(Y^{ab}) > x\,|\,\Z=\z) \;&\leq\; \inf_{t>0} \;\exp\bigl(- xt  \bigr)\prod_{1\leq i<j\leq n}\E\bigl( \exp[ t(Y^{ab}_{ij}-\E(Y^{ab}_{ij}))]\bigr)
\end{align*}
where the expectations are given $\Z=\z$.  Define
\begin{align*}
\mathds{1}_{ij}(\e,\z) \;=\; & \mathds{1}\{e_i=a,e_j=b\}+ \mathds{1}\{e_i=b,e_j=a\} \\
&- \mathds{1}\{z_i=a,z_j=b\} - \mathds{1}\{z_i=b,z_j=a\}\,.
\end{align*}
We have
\[
\E(Y_{ij}^{ab}\mid \Z=\z) = \frac{\rho_nS_{z_iz_j}\mathds{1}_{ij}(\e,\z)}{n^2}\,.
\]
On the other hand
\begin{align*}
\E(\exp(tY^{ab}_{ij})\mid \Z=\z) \;&=\; \rho_nS_{z_iz_j} \exp\Bigl(\frac{t\mathds{1}_{ij}(\e,\z)}{n^2}\Bigr)+  1 - \rho_nS_{z_iz_j}\\
&=\;  1 + \rho_nS_{z_iz_j} \Bigl[ \exp\Bigl(\frac{t\mathds{1}_{ij}(\e,\z)}{n^2}\Bigr) - 1\Bigr]\\
&\leq \; \exp\Bigl(  \rho_nS_{z_iz_j} \Bigl[\exp\Bigl(\frac{t\mathds{1}_{ij}(\e,\z)}{n^2}\Bigr) - 1\Bigr]\Bigr)\,.
\end{align*}
Then 
\begin{align*}
\E(\exp(&t(Y^{ab}_{ij}- \E(Y^{ab}_{ij}\mid \Z=\z))\mid \Z=\z) \;\leq\;\exp\Bigl(  \rho_nS_{z_iz_j} L\Bigl( \frac{t\mathds{1}_{ij}(\e,\z)}{n^2}\Bigr) \Bigr)
\end{align*}
with $L(y)=e^y-y-1$.
Observe that 
\begin{equation}\label{eq:sum_1ij}
\begin{split}
\sum_{1\leq i<j\leq n}\mathds{1}_{ij}(\e,\z)  = \sum_{i,j=1}^n\mathds{1}\{e_i=a,e_j=b\} - \mathds{1}\{z_i=a,z_j=b\}.
\end{split}
\end{equation}
For fixed $a,b\in[k]$, the difference of the indicator functions in the summand of \eqref{eq:sum_1ij} can be nonzero only if at least one of the vertices $i$ or $j$ is misclassified. Let $M=\{i:e_i\neq z_i\}$ with $|M|=m$. Since the number of pairs $(i,j)$ such that $i\in M$ or $j\in M$ is at most $2nm$ and each term has absolute value at most $1$, we obtain
\begin{equation}
\sum_{1\leq i<j\leq n}|\mathds{1}_{ij}(\e,\z)| \leq 2nm.
\end{equation}

 Then, taking  $t=n^2$  and using that $L(y)\leq 3|y|$ for $y\in [-2,2]$ we have that
\begin{align*}
\P( Y^{ab}& - \E(Y^{ab})  > x\,|\,\Z=\z) \;\\
&\leq\; \exp\bigl(- xn^2 \bigr) \prod_{1\leq i<j\leq n}\exp\Bigl[3\rho_nS_{z_iz_j}|\mathds{1}_{ij}(\e,\z)|\Bigr]\\
&\leq\; \exp\Bigl(- xn^2 + 6\rho_n  \sma  m(\e,\z) n\Bigr).
\end{align*}
The same can be obtained for left deviations. Therefore 
\begin{equation*}
\P( |W_{ab}(\e,\z)| > x\,|\,\Z=\z) \;\leq\;  2 \exp\Bigl(- xn^2 + 6\rho_n  \sma  m(\e,\z) n\Bigr)\,. \qedhere
\end{equation*}
\end{proof}

\begin{corollary}\label{corrate}
For  $\rho_n\geq \log n/n$, there exists a constant $c>0$, only depending on $S$, such that, given $\Z=\z$
\begin{equation}\label{Waball}
|W_{ab}(\e,\z)|\; \leq\;   \frac{c \rho_n m(\e,\z)}{n}
\end{equation}
simultaneously for all $a,b\in[k]$ and $\e\in[k]^n$, eventually almost surely as $n\to\infty$. 
Moreover, if $\z$ satisfies that $n_{a}(\z) \geq \pi_{\min}n/2$ for all $a\in[k]$, we have that 
\begin{equation}\label{oes}
\Bigl| \dfrac{\o{a}{b}(\e)}{\n{a}{b}(\e)} - \dfrac{\o{a}{b}(\z)}{\n{a}{b}(\z)}\Bigr| \;\leq\;  \frac{c \rho_n m(\e,\z)}{n}
\end{equation}
simultaneously for all $a,b\in[k]$ and $\e\in[k]^n$ with $m(\z,\e)\leq \pi_{\min}^2n/16$, eventually almost surely as $n\to \infty$.
\end{corollary}
\begin{proof}
We use the concentration of Proposition~\ref{lemma-chernoff2} and we use  a union bound over the number of differences $m$ and the community assignments $\e\in [k]^n$ such that $m(\e,\z)=m$ to obtain that 
\begin{align*}
\P\Bigl( \; \sup_{a,b,m,\e} &|W_{ab}(\e,\z)| > x\,|\,\Z=\z\Bigr) \;\\
&\leq\; 2\sum_{a,b\in [k]}\sum_{m=1}^n \binom{n}{m} (k-1)^m \exp\Bigl(- xn^2 + 6\rho_n  \sma  m n\Bigr)\\
&\leq\; 2k^2\sum_{m=1}^n\exp\Bigl(- xn^2 + 6\rho_n  \sma  m n+ m\log (k-1)n\Bigr).
\end{align*}
Taking  $x = c\rho_n m/n$,  we can bound  
the probability above by 
\[
2k^2\sum_{m=1}^n\exp\Bigl[ -m\Bigl( c\rho_nn  - 6\sma  \rho_n  n - \log (k-1)n\Bigr)\Bigr]\,.
\]
If $\rho_n n \geq \log n$ and  $c> 6\sma$, the series above is summable over $n$, therefore we conclude that 
\begin{equation*}
|W_{ab}(\e,\z)|\; \leq\;   \frac{c \rho_n m(\e,\z)}{n}
\end{equation*}
simultaneously for all $a,b\in[k]$ and $\e\in[k]^n$, eventually almost surely as $n\to\infty$. 
This proves the first statement of Corollary~\ref{corrate}. To prove \eqref{oes}, we add and substract  $o_{ab}(\z)/n_{ab}(\e)$ and apply the triangle  inequality, to obtain  that
\begin{equation}\label{eq:cor_diff}
\begin{split}
\Bigl| \dfrac{\o{a}{b}(\e)}{\n{a}{b}(\e)} - \dfrac{\o{a}{b}(\z)}{\n{a}{b}(\z)}\Bigr| \;\leq\;&
\frac{1}{\n{a}{b}(\e)}\Bigl| \o{a}{b}(\e)- \o{a}{b}(\z)\Bigr|\\
&+ \frac{ \o{a}{b}(\z)}{\n{a}{b}(\z)\n{a}{b}(\e)}\Bigl| \n{a}{b}(\z)- \n{a}{b}(\e)\Bigr|\,.
\end{split}
\end{equation}
To bound the first term in the sum on the right-hand side of \eqref{eq:cor_diff}, we add and subtract the corresponding expected values $\E(o_{ab}(\e))$ and $\E(o_{ab}(\z))$, apply the triangle inequality, and combine \eqref{Waball} with the fact that
\[
 \Bigl| \dfrac{\E(\o{a}{b}(\e))}{n^2} - \dfrac{\E(\o{a}{b}(\z))}{n^2}\Bigr| \;\leq\;  \frac{3\sma\rho_n m(\e,\z)}{n} 
\]
obtaining the upper bound
 \begin{align*}
\frac{c'\rho_n m(\e,\z)n}{\n{a}{b}(\e)}
\end{align*}
with $c'>0$ a given constant only depending on $S$. On the other hand,  by Theorem~\ref{teo_basico1} we have that $o_{ab}(\z)/n_{ab}(\z) = o(\rho_n)$ simultaneously for all $a,b\in[k]$, eventually almost surely as $n\to\infty$. Combining this with the fact that $|\n{a}{b}(\z) - \n{a}{b}(\e)| \leq 2nm(\e,\z)$, we conclude that the second term on the right-hand side of \eqref{eq:cor_diff} is bounded above by 
\[
\frac{c''\rho_n m(\e,\z)n}{\n{a}{b}(\e)}.
\]
Observe that $\n{a}{b}(\e)\geq \n{a}{b}(\z)-2nm(\e,\z)\geq \pi_{\min}^2n^2/8$  for any $\e\in[k]^n$ with $m(\z,\e)\leq \pi_{\min}^2n/16$. Thus,
\begin{equation}
\Bigl| \dfrac{\o{a}{b}(\e)}{\n{a}{b}(\e)} -  \dfrac{\o{a}{b}(\z)}{\n{a}{b}(\z)}\Bigr| \;\leq\;  \frac{c\rho_n m(\e,\z)}{n}
\end{equation}
simultaneously for all $a,b\in[k]$ and $\e\in[k]^n$ with $m(\z,\e)\leq \pi_{\min}^2n/16 $, eventually almost surely as $n\to\infty$. 
\end{proof}

 The next proposition is derived here only on the sparse case $\rho_n\to 0$. The version for the dense case is given in Proposition \ref{lemma-chernoff_denso}.

\begin{proposition}\label{lemma-chernoff}
Let $(\Z,\A)$ be generated from a SBM with parameters $(\pi, \rho_n S)$, with $\rho_n\to0$ as $n\to\infty$. Then,
for all $\z, \e \in [k]^n$  and for 
\begin{align*}
W(\e,\z) = & \frac12 \sum_{1\leq a, b\leq k}\, W_{ab}(\e,\z)\log \frac{ \rho_n S_{ab}}{1-\rho_n S_{ab}}
\end{align*}
we have that 
\begin{equation*}
\begin{split}
 \P\Bigl(   W(\e,\z)   > x & \mid \Z=\z   \Bigr) \\
 &\leq \exp\Bigl[ - \sup_{t\in (0,1]} \Bigl\{ n^2 (xt-  \rho_n C_t(R,S))\Bigl\}  +  ck^2\rho_n^2 nm  + D\rho_nm^2\Bigr]\,,
\end{split}
\end{equation*}
with  
\begin{equation}\label{tildeC}
C_t(R,S) = \sum_{a}\sum_{b\neq b'} [R^T\1]_{a}  K_t(S_{ab} \| S_{ab'})R_{b'b}  \,,
\end{equation}
and 
\begin{equation}\label{Kt}
K_t(p\|q) =   p^{1-t}q^t + t p \log \Bigl(\frac{p}{q}\Bigr) - p
\end{equation}
where $c$ and $D$ are constants only depending on $S$.
\end{proposition}

\begin{proof}
We consider the sparse case with $\rho_n \to 0$ as $n \to \infty$. In this regime, $\log(1-\rho_n S_{ab})$ is asymptotically negligible and of order $\rho_n$, so by Corollary~\ref{corrate} we have that
\[ 
\Bigl|  \frac12 \sum_{1\leq a, b\leq k}\, W_{ab}(\e,\z)\log (1 - \rho_nS_{ab}) \Bigr| \leq \frac{ck^2\rho_n^2 m}{n}
\]
for some constant $c>0$, simultaneously for all $a,b\in[k]$ and $\e\in[k]^n$, eventually almost surely as $n\to\infty$.  Thus, we have that
\[
W(\e,\z)\leq \frac12 \sum_{1\leq a, b\leq k}\, W_{ab}(\e,\z)\log S_{ab} + \frac{ck^2 \rho_n^2m}{n} 
\]
as
\[
 \sum_{1\leq a, b\leq k}\, W_{ab}(\e,\z) = 0.
\]

By \eqref{Wab},  we can define 
\begin{equation*}
\begin{split}
W'(\e,\z) =  \frac12 \sum_{1\leq a, b\leq k}\, W_{ab}(\e,\z)\log S_{ab} =  \frac1{2n^2} \sum_{1\leq a, b\leq k}\,\,\sum_{1\leq i,j\leq n} [ Y_{ij}^{ab} - \E(Y_{ij}^{ab}) ]
\end{split}
\end{equation*}
with 
\[
Y_{ij}^{ab} = \log  S_{ab}\Bigl( \mathds{1}\{e_i=a,e_j=b\}A_{ij}\, - \mathds{1}\{z_i=a,z_j=b\}A_{ij} \Bigr)\,.
\]
Let  
\[
Y_{ij} = \sum_{1\leq a, b\leq k}Y_{ij}^{ab} .
\]
Then it can be written as
\[
Y_{ij} = \log\Bigl( \frac{S_{e_ie_j}}{S_{z_iz_j}}\Bigr) A_{ij}\,.
\]
Conditioned on $\Z=\z$, the variables $Y_{ij}$, for $1\leq i<j\leq n$ are independent, $Y_{ji} = Y_{ij}$ and $Y_{ii}=0$ for all $i=1,\dots,n$. 
Then applying Chernoff's bound, we  obtain that 
\begin{align*}
\P( W'(\e,\z) > x\mid &\Z=\z) =  \P\Bigl( \,\sum_{1\leq i<j\leq n} Y_{ij} - \E(Y_{ij}) > n^2x\,|\,\Z=\z\Bigr) \\
&\leq\; \inf_{t\in (0,1]} \;\Bigl\{ \exp\bigl(- n^2xt  \bigr)\prod_{1\leq i<j\leq n}\E\bigl( \exp[ t(Y_{ij}-\E(Y_{ij}))]\bigr)\Bigr\}
\end{align*}
where the expectations are given $\Z=\z$. First, observe that
\[
\E(tY_{ij}\mid \Z=\z) = t\rho_n S_{z_iz_j}\log\Bigl( \frac{S_{e_ie_j}}{S_{z_iz_j}}\Bigr)\,.
\]
On the other hand, using that $e^x \geq 1+x$, we can write
\begin{align*}
\E(\exp(tY_{ij})\mid \Z=\z) &= \rho_n S_{z_iz_j} \exp\Bigl(t\log\Bigl( \frac{S_{e_ie_j}}{S_{z_iz_j}}\Bigr)
\Bigr)+  1 - \rho_n S_{z_iz_j}\\
&=\;  1 + \rho_n S_{z_iz_j} \Bigl[ \exp\Bigl(t\log\Bigl( \frac{S_{e_ie_j}}{S_{z_iz_j} }\Bigr)
\Bigr) - 1\Bigr]\\
&\leq \; \exp \Bigl\{   \rho_n S_{z_iz_j}\Bigl[ \exp\Bigl(t\log\Bigl( \frac{S_{e_ie_j}}{S_{z_iz_j}}\Bigr)
\Bigr) - 1\Bigr]\Bigr\}\,.
\end{align*}
Combining the two expectations, we obtain
\begin{align*}
\E(\exp[t (Y_{ij} - \E(Y_{ij}\mid \Z=\z)]\mid \Z=\z) \;&\leq\; \exp\Bigl\{  \rho_n S_{z_iz_j} L\Bigl[t\log\Bigl( \frac{S_{e_ie_j}}{S_{z_iz_j}}\Bigr)\Bigr]\Bigr\}\,,
\end{align*}
with $L(x) = e^x-x-1$. Therefore
\begin{equation}\label{ineqchernoff}
\begin{split}
\P( \,W'(&\e,\z) > x\,|\,\Z=\z) \;\\
&\leq\; \;\inf_{t\in (0,1]} \Bigl\{ \;\exp\Bigl(  - n^2x t + \sum_{1\leq i<j\leq n}  \rho_n S_{z_iz_j} L\Bigl[t\log\Bigl( \frac{S_{e_ie_j}}{S_{z_iz_j}}\Bigr)\Bigr]\Bigr)\Bigr\}\\
&= \;\exp\Bigl( - \sup_{t\in(0,1]} \; \Bigl\{ n^2x  t -  \rho_n\sum_{1\leq i<j\leq n}    S_{z_iz_j} L\Bigl[t\log\Bigl( \frac{S_{e_ie_j}}{S_{z_iz_j}}\Bigr)\Bigr]\Bigr\}\Bigr)\,.
\end{split}
\end{equation}
Observe that for all $1\leq i<j\leq n$ we have 
\begin{equation*}
\begin{split}
S_{z_iz_j} L\Bigl[t\log\Bigl( \frac{S_{e_ie_j}}{S_{z_iz_j}}\Bigr)\Bigr] \;&=\;  S_{z_iz_j}^{1-t} S_{e_ie_j}^t+t S_{z_iz_j}\log \Bigr(\frac{S_{z_iz_j}}{S_{e_ie_j}}\Bigr) -  S_{z_iz_j}\\
&= K_t(S_{z_iz_j} \| S_{e_ie_j}),
\end{split}
\end{equation*}
with $K_t$ defined by \eqref{Kt}. 
Using that matrix $S$ is symmetric and denoting $R=R(\e,\z)$ we can write 
\begin{equation}\label{Ktineq}
\begin{split}
\sum_{1\leq i<j\leq n}K_t(S_{z_iz_j} \| S_{e_ie_j})\;&=\;  \frac{n^2}2\;\sum_{a',a}\sum_{b',b} R_{a'a}K_t(S_{ab} \| S_{a'b'})R_{b'b}\\
&=\;  n^2\;\sum_{b\neq b'}\sum_{a} R_{aa}K_t(S_{ab} \| S_{ab'})R_{b'b} \\
&\quad + \frac{n^2}2\;\sum_{a\neq a'}\sum_{b\neq b'} R_{a'a}K_t(S_{ab} \|S_{a'b'})R_{b'b}.
\end{split} 
\end{equation}
Using the fact that $m/n = \sum_{a\neq a'}R_{a'a}$ and that $K_t(S_{ab},S_{a'b'})\leq D$ for all $a,b,a',b'\in [k]$, where $D$ is a constant depending only on $S$, we observe that the second sum on the right-hand side of \eqref{Ktineq} is bounded above by $m^2D$. Using that $R_{aa}\leq [R^T\one]_a$ for all $a\in[k]$, we obtain from \eqref{ineqchernoff} and \eqref{Ktineq} that
\begin{equation*}
\begin{split}
\P\Bigl(W(\e,\z) >x & \mid \Z=\z\Bigr) \leq \P\Bigl(W'(\e,\z) >x - \frac{ck^2\rho_n^2m}{n} \,|\,\Z=\z\Bigr) \\
& \leq \;\exp\Bigl[ - \sup_{t\in (0,1]}  \Bigl\{n^2 \bigl(xt-   \rho_n C_t(R,S)\bigr)\Bigr\} + ck^2\rho_n^2 nm + D\rho_n m^2 \Bigr]
\end{split}
\end{equation*}
with $C_t(R,S)$ given by 
\begin{equation}
C_t(R, S) = \sum_{a}\sum_{b\neq b'} [R^T\1]_{a}  K_t(S_{ab} \| S_{ab'})R_{b'b} 
\end{equation}
and $K_t$ given by \eqref{Kt}.
\end{proof}

\bibliographystyle{authordate1}
\bibliography{references}

\newpage

\noindent{\large \bf Suppementary material}


\section{General Auxiliary Results}\label{sec:modularities_relation}

In this section, we present the proof of Lemma~\ref{lemma-qml-qb}, relating the maximum and integrated conditional likelihood modularities, and we present other auxiliary results that are necessary to prove the weak and strong consistency of the proposed estimators.  

\begin{proof}[Proof of Lemma \ref{lemma-qml-qb}]
Observe that for any $\e\in [k]^n$ we have 
\begin{equation*}\label{razao_a|z}
\Q{\e}-\Qb{\e} \;=\; 
\frac{1}{n^2}\sum\limits_{1\leq a \leq b \leq k}  \log  Q(\tilde o_{ab}(\e,\a), \tilde n_{ab}(\e))  
\end{equation*}
with 
\[
 Q(o_{ab}(\e,\a), n_{ab}(\e))  \;=\;  \dfrac{ \Gamma\bigl(\frac12\bigr)^2\Gamma\Bigl(\tilde n_{ab}+1\Bigr)\Bigl( \dfrac{\tilde o_{ab}}{ \tilde n_{ab}} \Bigr)^{ \tilde o_{ab}}\Bigl(1-\dfrac{\tilde o_{ab}}{\tilde n_{ab}}\Bigr)^{ \tilde n_{ab}-\tilde o_{ab}}}{\Gamma\Bigl(\tilde o_{ab}+\frac{1}{2}\Bigr)\Gamma\Bigl( \tilde n_{ab}-\tilde o_{ab}+\frac{1}{2}\Bigr)}\,,
\]
where, for simplicity in the notation, we have omitted the argument $\e$ and $\a$ from the counts.
For each pair $a,b$, $1\leq a,b \leq k$, 
using Lemma 5.1 in \cite{cerqueira2024}, we get
\begin{equation*}
\begin{split}
\dfrac{\left( \dfrac{\tilde o_{ab}}{ \tilde n_{ab}} \right)^{ \tilde o_{ab}}\left(1-\dfrac{\tilde o_{ab}}{\tilde n_{ab}}\right)^{ \tilde n_{ab}-\tilde o_{ab}}}{\Gamma\left(\tilde o_{ab}+\frac{1}{2}\right)\Gamma\left( \tilde n_{ab}-\tilde o_{ab}+\frac{1}{2}\right)}\;\leq\; \frac{1}{\Gamma\left( \tilde n_{ab}+\frac{1}{2}\right)\Gamma\left(\frac{1}{2}\right)}\,.
\end{split}
\end{equation*}
Then  
\begin{equation*}
\begin{split}
& \Q{\e}-\Qb{\e}\; \leq\; 
\frac{1}{n^2}\sum\limits_{1\leq a \leq b \leq k}  \log\biggl( \frac{\Gamma(\frac12)\Gamma\left( \tilde n_{ab}+1\right)}{\Gamma\left( \tilde n_{ab}+\frac{1}{2}\right)} \biggr)\,.
\end{split}
\end{equation*}
Using the standard bound for the Gamma function
\begin{equation*}
x^{x-\frac{1}{2}}e^{-x}\sqrt{2\pi} \;\leq\; \Gamma(x)\; \leq\; x^{x-\frac{1}{2}}e^{-x}\sqrt{2\pi} e^{\frac{1}{12x}}
\end{equation*}
valid for $x>0$ we have that
\begin{equation*}
\begin{split}
&\log\biggl( \frac{\Gamma(\frac12)\Gamma\left( \tilde n_{ab}+1\right)}{\Gamma\left( \tilde n_{ab}+\frac{1}{2}\right)} \biggr)\; \leq\; (\tilde n_{ab}+\textstyle{\frac{1}{2}})\log\left(\tilde n_{ab}+1\right) 
+\frac{1}{12(\tilde n_{ab}+1)}
-n_{ab}\log\left(\tilde n_{ab}+\frac{1}{2}\right)+\frac{1}{2}\\
&\qquad\quad\leq \;  \textstyle\frac{1}{2} \log \tilde n_{ab} + (\tilde n_{ab}+\textstyle\frac{1}{2})\log(1+\frac{1}{\tilde n_{ab}})-\tilde n_{ab}\log(1+\frac{1}{2\tilde n_{ab}})+\textstyle{\frac{1}{2}+\frac{1}{12 \tilde n_{ab}}} \\
& \qquad\quad\leq\; \log n + 2\,,
\end{split}
\end{equation*}
where the last inequality follows by $1-\frac{1}{x} \leq \log x \leq x-1$ and $1\leq \tilde n_{ab} \leq n^2$. Using that $\mathbb Q(\a\mid \z) \leq \max_{P}\P(\a\mid \z)$ we obtain
\[
\Q{\e}-\Qb{\e} \;\geq\; 0
\]
and we conclude that 
\begin{equation*}
 |\,  \Q{\e}-\Qb{\e} \,|  \;\leq\; 
\frac{k(k+1)(\log n+2)}{2n^2} \;\leq\; 
\frac{k^2(\log n+2)}{n^2}\,.\qedhere
\end{equation*}
\end{proof}

The following two lemmas were originally stated and proved in \citet{van2018bayesian} and are included here with minor adaptations for completeness.

\begin{lemma}\label{lemma-basico-vandpass}
For every $\e,\z \in[k]^n$ we have that 
\[
\frac1n\sum_{i=1}^n \1\{e_i\neq z_i\} \;=\; \frac12\|\diag(R(\e,\z)^T\1) - R(\e,\z)\|_1\,,
\]
where $\diag(R(\e,\z)^T\1)$ is a $k\times k$ diagonal matrix with entries $[R(\e,\z)^T\1]_a$ on the diagonal for all $a\in [k]$. 
\end{lemma}

\begin{proof}
See Lemma~1 in \citet{van2018bayesian}. 
\end{proof}

\begin{lemma}\label{lemma-basico-vandpass0}
 Let $P=\rho_n S$, with $S$ fixed and $\rho_n=1$ or $\rho_n\to0$ as $n\to\infty$. Then if $\rho_n=1$ we have that
 \begin{equation}\label{sma1}
H_{P}(\diag(R^T\1)) - H_{P}(R)) \geq 0 
\end{equation}
and if $\rho_n\to\infty$
 \begin{equation}\label{sma3}
\frac{1}{\rho_n} \biggl( H_{\rho_n S}(\diag(R^T\1))  - H_{\rho_n S}(R))\biggr)\;\rightarrow\;G_{S}(\diag(R^T\1)) - G_{S}(R))\,,
\end{equation}
with
\[
G_S(R) \;=\; \sum_{a,b} [RSR^T]_{ab}  \gamma\Bigl( \frac{ [RSR^T]_{ab}}{ [R\1]_a [R\1]_b}\Bigr) 
\]
for $\gamma(x) = x\log(x) - x$. 
Furthermore, if $(S,\pi)$ is identifiable and the columns of $R$ corresponding to positive coordinates of $\pi$ are not identically zero, then the inequality in \eqref{sma1} is strict and the right hand side of \eqref{sma3} is also strictly positive, unless $D_\sigma R$ is a diagonal matrix for some permutation matrix $D_\sigma$.
\end{lemma}

\begin{proof}
The proof of \eqref{sma1} is given in \citet[Lemma 7]{van2018bayesian}. 
%
The proof of \eqref{sma3} is given in \citet[Lemma~8]{van2018bayesian}, but it holds for a matrix $S$ with values in $[0,1]$. But the statement is valid for any 
bounded matrix $S$, by considering $S'= \sma^{-1}S$ with vaues in $[0,1]$ and taking $\rho_n'= \sma \rho_n\to 0$ as $n\to\infty$. Then the resut is valid for $P=\rho_n'S' = \rho_n S$. 
\end{proof}

In the next result, we show that the functions $H_{P,n}$ and $H_P$ differ by a negligible term.

\begin{lemma}
\label{lem:hps}
If $\rho_n n \to\infty$ as $n\to\infty$, then for any $\delta>0$ we have that 
\[
|H_{P,n}(R) - H_P(R)| \leq \delta\rho_n
\]
uniformily over matrices $R$ with $[R\one]_a\geq (\rho_nn)^{-1}$, for $n$ sufficiently large. 
\end{lemma}

\begin{proof}
For any matrix $R\in \mathbb R^{k\times k_0}_{\geq 0}$
we  write 
\begin{align*}
H_{P,n}(R) &= \frac 1 2\sum_{1\leq a,b\le k} [R\one]_a \Bigl(  [R\one]_b -  \frac{\delta_{ab}}{n} \Bigr)\tau([P_R]_{ab})
\end{align*}
with $\tau(x) = x\log(x) + (1-x)\log(1-x)$, $x\in[0,1]$, $[P_R]_{ab}$ given by \eqref{pmixt}
and
\begin{align*}
H_{P}(R) &= \frac 1 2\sum_{1\leq a,b\le k} [R\one]_a [R\one]_b \tau\left(\frac{[ RPR^\intercal]_{ab}}{[R\one]_a[R\one]_b}\right).
\end{align*}
Then, we obtain that 
\begin{equation}
\begin{split}
H_{P,n}(R) - H_{P}(R) &=  \frac 1{2} \sum_{1\leq a \leq k}  [R\one]^2_a \left( \tau\left([P_R]_{aa}\right)-\tau\left(\frac{[ RPR^\intercal]_{aa}}{[R\one]_a[R\one]_a}\right)\right)\\
&\quad + \frac{1}{n}\sum_{1\leq a \leq k}  [R\one]_a  \tau\left([P_R]_{aa}\right).
\end{split}
\end{equation}
As $\tau$ is a bounded function and $n^{-1}\ll \rho_n$ we obtain that the second term can be bounded by $\delta\rho_n$ uniformily on $R$, for any $\delta>0$. For the first term we use  the Taylor expansion of the function $\tau$ when $\rho_n\to0$, given by 
\begin{equation*}
\tau(\rho_n x) \;=\; \rho_n \gamma(x) +  x \rho_n \log\rho_n + O\bigl(\rho_n^2x^2\bigr)
\end{equation*}
with
\[
\gamma(x) \;=\; x\log x - x\,.
\]
Using that $[P_R]_{aa}$ and ${[ RPR^\intercal]_{aa}}/[R\one]_a[R\one]_a$ belong to the interval  $(\rho_n\smi,\rho_n\sma)$ and $\gamma$ is Lipschitz on any bounded interval, there exists a constant $0< M< \infty$ such  that 
\begin{equation}\label{eq:dif_tau_H}
\begin{split}
\Bigl| \tau\left([P_R]_{aa}\right)-&\tau\left(\frac{[ RPR^\intercal]_{aa}}{[R\one]_a[R\one]_a}\right) \Bigr| \\
&\leq \bigl( M\rho_n + \rho_n \log \rho_n \bigr) \biggl|  \frac{[RSR^\intercal]_{aa}- \sum_{b}S_{bb}R_{ab}/n}{[R\one]_a([R\one]_a-1/n)}- \frac{[ RSR^\intercal]_{aa}}{[R\one]_a[R\one]_a} \biggr| .
\end{split} 
\end{equation}
As $\rho_n\log\rho_n \to0 $ as $n\to\infty$, for any $\delta>0$ we obtain that 
\[
M\rho_n + \rho_n \log \rho_n <\delta
\]
for $n$ sufficiently large. Moreover, by summing and subtracting the term 
\[
\frac{[RSR^\intercal]_{aa}}{[R\one]_a([R\one]_a-1/n)}
\]
to the expression inside the absolute value on the right-hand side of \eqref{eq:dif_tau_H}, we can bound 
\begin{equation}
\begin{split}
\biggl|  \frac{[RSR^\intercal]_{aa}- \sum_{b}S_{bb}R_{ab}/n}{[R\one]_a([R\one]_a-1/n)}- \frac{[ RSR^\intercal]_{aa}}{[R\one]_a[R\one]_a} \biggr| \;&\leq\; \frac{2\sma}{n([R\one]_a-1/n)}\\
& \leq\; C \rho_n,
\end{split}
\end{equation}
with $C<\infty$, provided $[R\one]_a\geq  (\rho_n n)^{-1}$ for all $a\in[k]$. Then we  obtain for any $\delta>0$ that 
\[
| H_{P,n}(R) - H_{P}(R) | \leq \delta\rho_n
\]
uniformly over matrices with $[R\one]_a\geq (\rho_n n)^{-1}$ for all $a\in[k]$, for $n$ large enough.
\end{proof}

In the work of \citet{van2018bayesian}, it was sufficient for the authors to show that the limit of the difference in $H_{\rho_n S, n}$, computed for the diagonal matrix $\operatorname{diag}(R^\top \mathbf{1})$ and the confusion matrix $R$, is bounded below by a positive constant. In our case, since we are interested in obtaining the phase transition threshold for strong consistency, it is necessary to derive an explicit formula for this constant. The following result extends Lemma 9 from \citet{van2018bayesian} to derive an explicit expression for the lower bound.

\begin{lemma}\label{lemma-rate-nosso}
Let $P=\rho_n S$ be symmetric, with $S>0$.  Let $\epsilon>0$ and $R$ such that $\|R-\diag(\pi)\|_1 < \delta$, with $\delta$ sufficiently small. 
If  $\rho_n=1$ we have that
\begin{equation}\label{vdpeq1}
\begin{split}
H_{P,n}(\diag(R^T\1)) - H_{P,n}(R) \;\geq \;&\sum_{b\neq b'}  \Bigl(  \sum_{a} [R^T\1]_{a} KL(P_{ab}\| P_{ab'}) - \epsilon\Bigr)R_{b'b} ,
\end{split}
\end{equation}
where $KL$ is the \emph{Kullback-Leibler} divergence defined by $KL(p\|q) = p \log (\frac{p}{q}) + (1-p)\log(\frac{1-p}{1-q})$. If $\rho_n\to0$ as $n\to\infty$ we obtain the  bound 
\begin{equation}\label{vdpeq2}
\begin{split}
H_{\rho_n S,n}(\diag(R^T\1)) - H_{\rho_n S,n}(R) \;\geq \;& \rho_n \sum_{b\neq b'}  \Bigl(  \sum_{a} [R^T\1]_{a} K_1(S_{ab}\| S_{ab'}) - \epsilon\Bigr)R_{b'b}  
\end{split}
\end{equation}
with $K_1(p\|q) = p \log (\frac{p}{q}) +  q - p$.
\end{lemma}

\begin{proof}[Proof of Lemma~\ref{lemma-rate-nosso}]
By a lengthy calculation, \citet[Lemmas 9-10]{van2018bayesian} proved that the left hand side of \eqref{vdpeq1} equals
\[
- \nabla G(0)^T \lambda - \int_0^1 ( \nabla G(s\lambda) -  \nabla G(0))^T ds\, \lambda
\]
with $\lambda$ a vector defined by $\lambda_{b'b}= R_{b'b}$ for $b\neq b'$ and 
\[
G(\lambda)\;=\; H_{P,n}\Bigl(  \diag(R^T\1) + \sum_{b\neq b'}\lambda_{b'b}\Delta_{b'b}    \Bigr)\,
\]
for some given  matrices $\Delta_{b'b}$. Then the authors proved that 
\[
\frac{\partial}{\partial \lambda_{b'b}} G(\lambda)|_{\lambda=0} \;=\; - \sum_{a} [R^T\1]_{a} KL(P_{ab}\| P_{ab'}) + \frac{1}{2n}KL(P_{bb}\|P_{b'b'})\,
\]
for $KL(s\|t) = s\log (s/t) + (1-s)\log((1-s)/(1-t))$. Then 
\begin{align*}
- \nabla G(0)^T \lambda \;&=\;  \sum_{b\neq b'}  \Bigl(  \sum_{a} [R^T\1]_{a}KL(P_{ab}\| P_{ab'}) -  \frac{1}{2n}KL(P_{bb}\|P_{b'b'})\Bigr)R_{b'b}\\
&\geq\;  \sum_{a}\sum_{b\neq b'}  [R^T\1]_{a}KL(P_{ab}\| P_{ab'})R_{b'b} -  \frac{\max_{b\neq b'} KL(P_{bb}\|P_{b'b'}) }{2n} \sum_{b\neq b'} R_{b'b}.
\end{align*}
%
On the other hand,  as shown by \citet[Lemmas 9-11]{van2018bayesian}, as $\rho_n^{-1}\|\nabla G(\lambda)- \nabla G(0)\|$ becomes uniformly small as $\lambda$ is 
close enough to zero, for $\rho_n=1$ or $\rho_n\to 0$ as $n\to\infty$, then  
\begin{align*}
\Bigl| \int_0^1 ( \nabla G(s\lambda) -  \nabla G(0))^T ds\, \lambda \Bigr| \;&\leq\;  \epsilon \rho_n  \sum_{b\neq b'} R_{b'b}. 
\end{align*}
Therefore, by seeing that 
\[
\max_{b\neq b'} KL(P_{bb}\|P_{b'b'}) \leq C\rho_n
\]
with $C$ a constant not depending on $n$, and for $\rho_n=1$ or $\rho_n\to 0$ as $n\to\infty$, we obtain that 
\begin{align*}
H_{P,n}(\diag(R^T\1)) - H_{P,n}(R) \;\geq \;  & \sum_{b\neq b'} \Bigl(  \sum_{a}  [R^T\1]_{a}KL(P_{ab}\| P_{ab'}) -\epsilon\rho_n\Bigr)R_{b'b}.
\end{align*}
In the case $\rho_n\to0$ as $n\to\infty$ we have that 
\[
KL(\rho_n p\| \rho_n q) \;\geq \;  \rho_n K_1(p\|q) - c \rho_n^2
\]
with $K_1(p\|q)=p\log(p/q) + q - p$. Therefore for $P=\rho_nS$ with $\rho_n\to0$ as $n\to\infty$ we obtain that 
\begin{align*}
H_{\rho_n S,n}(\diag(R^T\1)) - H_{\rho_nS,n}(R) \;\geq \;  &\rho_n\sum_{b\neq b'}  \Bigl( \sum_{a}[R^T\1]_{a}K_1(S_{ab}\|S_{ab'}) - \epsilon \Bigr) R_{b'b}
\end{align*}
for $\epsilon$ arbitrarily small and $n$ sufficiently large. 
\end{proof}

\section{Concentration results for the dense case}\label{sec:dense}

In the dense case ($\rho_n=1$), Theorems \ref{theoremX2} and \ref{lemma-chernoff} are not applicable and must be adjusted accordingly. To do so, the key quantities $C_t(R,S)$ and $K_t$, obtained for the sparse case, are replaced by
\begin{equation}\label{tildeC_denso}
C'_t(R,S) = \sum_{a}\sum_{b\neq b'} [R^T\1]_{a}  D_t(S_{ab} \| S_{ab'})R_{b'b}\,,
\end{equation}
with 
\begin{equation}\label{Kt_denso}
D_t(p \|q) \;=\;   \log\left(p\left(\frac{q(1-p)}{p(1-q)} \right)^{t} +1-p\right)  + pt\log \frac{p(1-q)}{q(1-p)} ,
\end{equation}
for $p,q\in [0,1]$.
Observe that $D_1(p \|q)=KL(p\|q)$.

\begin{theorem}\label{theoremX2_denso}
Let $(\Z,\A)$ be generated from a SBM with parameters $(\pi, \rho_nS)$, with $\rho_n=1$. Let $\z\in [k]^n$ be such that  $n_{a}(\z) \geq \pi_{\min}n/2$ for all $a\in[k]$ and $\e\in \F$. Then, for all $\epsilon>0$, there exists $\tilde\epsilon>0$ such that, if   $m=m(\z,\e)\leq \tilde\epsilon n$ and  
$n$ is sufficiently large, we have that
\begin{align*}
\P( \{ X(\e) - X(\Z)  >  x\} &\cap \mathcal C \mid \Z=\z) \\
&\leq 
 \exp\Bigl[  - \sup_{t\in(0,1]} \Bigl\{  n^2\bigl(x t -  C'_t(R,S)\bigr)\Bigr\} +\epsilon mn\Bigr]
\end{align*}
with $C'_t(R,S)$ given by \eqref{tildeC_denso}.
\end{theorem}

\begin{proof}
Along the same lines as the proof of Theorem \ref{theoremX2}, the difference $X(\e) - X(\z)$  can be upper bounded by 
\begin{equation}\label{XXX_denso}
\begin{split}
 \frac12\sum_{1\leq a, b\leq k}\, &W_{ab}(\e,\z) \log  \frac{S_{ab}}{1-S_{ab}}\\
& + \frac12\sum_{1\leq a, b\leq k} \frac{1}{ S_{ab}} \; \Bigl|\E\Bigl(\dfrac{\o{a}{b}(\e)}{n^2}\Bigr) - \frac{\o{a}{b}(\e)}{n^2}\Bigr|\,
\Bigl| \dfrac{\o{a}{b}(\e)}{\n{a}{b}(\e)} - S_{ab}\Bigr|\\
& + \frac12\sum_{1\leq a, b\leq k} \frac{1}{1- S_{ab} } \; \Bigl|\E\Bigl(\dfrac{\o{a}{b}(\e)}{n^2}\Bigr) - \frac{\o{a}{b}(\e)}{n^2}\Bigr|\,
\Bigl| \dfrac{\o{a}{b}(\e)}{\n{a}{b}(\e)} - S_{ab}\Bigr|.
\end{split}
\end{equation}
Since the entries of $S$ are bounded, $S_{ab}$ and $1-S_{ab}$ are bounded below by a constant, for all $a,b\in [k]$. Then, we can bound \eqref{XXX_denso} from below by a sum of two terms 
$W(\e,\z) + \widetilde W(\e,\z)$ with 
\[
W(\e,\z) = \frac12 \sum_{1\leq a,b\leq k} W_{ab}(\e,\z)  \log  \frac{ S_{ab}}{1- S_{ab}}
\]
and
\[
\widetilde W(\e,\z) = C\sum_{1\leq a, b\leq k}  \; \Bigl|\E\Bigl(\dfrac{\o{a}{b}(\e)}{n^2}\Bigr) - \frac{\o{a}{b}(\e)}{n^2}\Bigr|\,
\Bigl| \dfrac{\o{a}{b}(\e)}{\n{a}{b}(\e)} - S_{ab}\Bigr|.
\]
The result follows along the same lines as Theorem \ref{theoremX2}, by Proposition \ref{lemma-chernoff2} and Proposition \ref{lemma-chernoff_denso}, proved in the sequel.
\end{proof}

\begin{proposition}\label{lemma-chernoff_denso}
Let $(\Z,\A)$ be generated from a SBM with parameters $(\pi, \rho_nS)$ with $\rho_n=1$. Then,
for all $\z, \e \in [k]^n$  and for 
\begin{align*}
W(\e,\z) = & \frac12 \sum_{1\leq a, b\leq k}\, W_{ab}(\e,\z)\log \frac{S_{ab}}{1-S_{ab}}
\end{align*}
we have that 
\begin{equation*}
\begin{split}
 \P\Bigl(\,  & W(\e,\z)  \; >\; x\,|\,\Z=\z  \, \Bigr) \;\leq\;  \exp\Bigl[ - \sup_{t\in (0,1]} \Bigl\{ n^2 (xt-   C'_t(R,S))\Bigl\}  + Fm^2\Bigr]\,,
\end{split}
\end{equation*}
with  $C'_t(R,S)$ given by \eqref{tildeC_denso} and $F$ a constant only depending on $S$.
\end{proposition}

\begin{proof}
Observe that  we can write
\begin{equation*}
\begin{split}
W(\e,\z) = \frac1{2n^2} \sum_{1\leq a, b\leq k}\,\,\sum_{1\leq i,j\leq n} [ Y_{ij}^{ab} - \E(Y_{ij}^{ab}) ]
\end{split}
\end{equation*}
with 
\[
Y_{ij}^{ab} = \log  \frac{S_{ab}}{1-S_{ab}}\Bigl( \mathds{1}\{e_i=a,e_j=b\}A_{ij}\, - \mathds{1}\{z_i=a,z_j=b\}A_{ij} \Bigr)\,.
\]
Taking
\[
Y_{ij} = \sum_{1\leq a, b\leq k}Y_{ij}^{ab} 
\]
we can write
\[
Y_{ij} = \log\Bigl( \frac{S_{e_ie_j}(1- S_{z_iz_j})}{S_{z_iz_j}(1- S_{e_ie_j}) }\Bigr) A_{ij}\,.
\]
Conditioned on $\Z=\z$, the variables $Y_{ij}$, for $1\leq i<j\leq n$ are independent, $Y_{ji} = Y_{ij}$ and $Y_{ii}=0$ for all $i=1,\dots,n$. 
Then applying Chernoff's bound, we  obtain that 
\begin{align*}
\P( \,W(\e,\z) > x \mid &\Z=\z)=  \P\Bigl( \,\sum_{1\leq i<j\leq n} Y_{ij} - \E(Y_{ij}) > n^2x\,|\,\Z=\z\Bigr) \\
&\leq\; \inf_{t\in (0,1]} \;\Bigl\{ \exp\bigl(- n^2xt  \bigr)\prod_{1\leq i<j\leq n}\E\bigl( \exp[ t(Y_{ij}-\E(Y_{ij}))]\bigr)\Bigr\}
\end{align*}
where the expectations are given $\Z=\z$. First, observe that
\[
\E(tY_{ij}\mid \Z=\z) = tS_{z_iz_j}\log\Bigl( \frac{S_{e_ie_j}(1- S_{z_iz_j})}{S_{z_iz_j}(1- S_{e_ie_j}) }\Bigr)\,.
\]
On the other hand
\begin{align*}
\E(\exp(tY_{ij})\mid \Z=\z) \;&=\; S_{z_iz_j} \exp\Bigl(t\log\Bigl( \frac{S_{e_ie_j}(1- S_{z_iz_j})}{S_{z_iz_j}(1- S_{e_ie_j}) }\Bigr)
\Bigr)+  1 - S_{z_iz_j}\\
&= S_{z_iz_j} \Bigl( \frac{S_{e_ie_j}(1- S_{z_iz_j})}{S_{z_iz_j}(1- S_{e_ie_j} }\Bigr)^t+  1 - S_{z_iz_j}\\
&= \exp\left\lbrace \log \left(  S_{z_iz_j} \Bigl( \frac{S_{e_ie_j}(1- S_{z_iz_j})}{S_{z_iz_j}(1- S_{e_ie_j} }\Bigr)^t+  1 - S_{z_iz_j}\right)  \right\rbrace.
\end{align*}
Combining the two expected values we obtain
\begin{align*}
\E(\exp[t (Y_{ij} - \E(Y_{ij}\mid \Z=\z)]\mid\Z=\z) &= D_t(S_{z_iz_j} \| S_{e_ie_j})
\end{align*}
where $D_t(p \|q)$ is given by \eqref{Kt_denso}, for $t \in (0,1]$ and $p,q>0$. Thus,
\begin{equation}\label{ineqchernoff_denso}
\begin{split}
\P( W(\e,\z) > x\mid\Z=\z) &\leq\inf_{t\in (0,1]} \Bigl\{ \exp\Bigl(  - n^2x t + \sum_{1\leq i<j\leq n} D_t(S_{z_iz_j} \| S_{e_ie_j})\Bigr)\Bigr\}\\
&= \exp\Bigl( - \sup_{t\in(0,1]} \; \Bigl\{ n^2x  t -  \sum_{1\leq i<j\leq n}D_t(S_{z_iz_j} \| S_{e_ie_j})  \Bigr\}\Bigr)\,.
\end{split}
\end{equation}

We have that $D_t(p\|q)\geq 0$ since the logarithm function is concave  and $D_t(p\|q)=0$ if and only if $p=q$. 
Using that matrix $S$ is symmetric and denoting $R=R(\e,\z)$ we can write 
\begin{equation}\label{Ktineq_denso}
\begin{split}
\sum_{1\leq i<j\leq n}D_t(S_{z_iz_j} \| S_{e_ie_j})\;&=\;  \frac{n^2}2\;\sum_{a',a}\sum_{b',b} R_{a'a}D_t(S_{ab} \| S_{a'b'})R_{b'b}\\
&=\;  n^2\;\sum_{b\neq b'}\sum_{a} R_{aa}D_t(S_{ab} \| S_{ab'})R_{b'b} \\
&\quad + \frac{n^2}2\;\sum_{a\neq a'}\sum_{b\neq b'} R_{a'a}D_t(S_{ab} \|S_{a'b'})R_{b'b}.
\end{split} 
\end{equation}
Now, the proof finishes by observing that $D_t(S_{ab},S_{ab'})\leq F$, with $F$ a constant only depending on $S$ and $R_{aa}\leq [R^T\one]_a$ for all $a\in[k]$.
Therefore by \eqref{ineqchernoff_denso} and \eqref{Ktineq_denso} we obtain that 
\begin{equation*}
\P\Bigl(W(\e,\z) >x\mid\Z=\z\Bigr) \leq \exp\Bigl[ - \sup_{t\in (0,1]}  \Bigl\{n^2 \bigl(xt-   C'_t(R,S)\bigr)\Bigr\} + F m^2 \Bigr]
\end{equation*}
with $C'_t(R,S)$ given by \eqref{tildeC_denso} with $D_t$ given by \eqref{Kt_denso}.
\end{proof}


\end{document}